\documentclass[12pt,a4paper]{article}
\usepackage[utf8x]{inputenc}
\usepackage[T1]{fontenc}
\usepackage[english]{babel}
\usepackage{amsmath}
\usepackage{amsfonts}
\usepackage{amssymb}
\usepackage{amsthm}
\usepackage{color}
\usepackage{graphicx}
\usepackage{subcaption}
\usepackage{afterpage}
\usepackage{mathtools}
\usepackage{pdfpages}
\usepackage{bm}
\usepackage{appendix}

\numberwithin{equation}{section}

\textheight 
            650pt
\textwidth 
           460pt

\newtheorem{theorem}{Theorem}[section]
\newtheorem{lemma}[theorem]{Lemma}
\newtheorem{proposition}[theorem]{Proposition}
\newtheorem{corollary}[theorem]{Corollary}
\theoremstyle{definition}
\newtheorem{definition}[theorem]{Definition}
\theoremstyle{remark}
\newtheorem{remark}[theorem]{Remark}

\newcommand{\curl}{\mathrm{curl}\,}
\newcommand{\divv}{\mathrm{div}\,}

\newcommand{\inn}{\text{in }}
\newcommand{\onn}{\text{on }}
\newcommand{\supp}{\mathrm{supp}}

\title{Steady three-dimensional ideal 
flows with nonvanishing vorticity
in domains with edges}
\author{Douglas S. Seth\thanks{Centre for Mathematical Sciences, Lund University, PO Box 118, 22100 Lund, Sweden; \text{douglas.svensson\_seth@math.lu.se}}}
\begin{document}
\maketitle
\begin{abstract}
We prove an existence result for solutions to the stationary Euler equations in a domain with nonsmooth boundary. This is an extension of a previous existence result in smooth domains by Alber (1992)\cite{Alber1992}. The domains we consider have a boundary consisting of three parts, one where fluid flows into the domain, one where the fluid flows out, and one which no fluid passes through. These three parts meet at right angles. An example of this would be a right cylinder with fluid flowing in at one end and out at the other, with no fluid going through the mantle. A large part of the proof is dedicated to studying the Poisson equation and the related compatibility conditions required for solvability in this kind of domain.
\end{abstract}
\begin{center}

\end{center}
\pagebreak
\section{Introduction}
A steady flow of an inviscid incompressible fluid through a simply connected domain $\Omega\subset \mathbb{R}^3$ satisfies the equations
\begin{alignat}{2}
(\bm{v}\cdot\nabla)\bm{v}+\nabla p&=0 &&\qquad \inn \Omega, \label{eq:Euler1a}\\
\divv \bm{v}&=0 &&\qquad \inn \Omega, \label{eq:Euler2}
\end{alignat}
and the boundary condition
\begin{equation}\label{eq:Eulerbdry}
\bm{v}\cdot \bm{n}=\phi \qquad \onn \partial \Omega,
\end{equation}
where $\bm{v}$ is the velocity field of the fluid, $p$ is the pressure, $\bm{n}$ is the outward normal of $\partial\Omega$ and $\phi$ is a given function satisfying
\begin{equation}\label{eq:phicompcond}
\int_{\partial\Omega}\phi(x)dS_x=0,
\end{equation}
see \cite{Marchioro1994}.
With the extra assumption that the velocity field is irrotational, i.e. that $\curl\bm{v}=0$, the complexity of the problem is reduced. Indeed, in this case, the velocity field is generated by a potential, which satisfies the Laplace equation with Neumann boundary conditions. This immediately gives a velocity field that satisfies \eqref{eq:Euler2} and \eqref{eq:Eulerbdry}. A quick calculation shows that \eqref{eq:Euler1a} is also satisfied given that the pressure is defined as $p=-\frac{1}{2}\vert \bm{v}\vert^2+C$ for some arbitrary constant $C$. The assumption that the flow is irrotational is, while mathematically convenient, often not physical as vorticity can be generated at rigid walls or by external forces, such as wind.

There are however also some existence results for rotational flows. Alber \cite{Alber1992} proved an existence result in smooth domains, which has provided much of the inspiration for this paper. The proof relies on adding two additional boundary conditions on the set where fluid enters the domain, which prescribe the vorticity there and make sure that the solution is rotational. The velocity field $\bm{v}$ is then split into a given solution to \eqref{eq:Euler1a}--\eqref{eq:Eulerbdry} and a small perturbation such that $\bm{v}$ satisfies \eqref{eq:Euler1a}--\eqref{eq:Eulerbdry} and the additional boundary conditions if and only if the perturbation is a fixed point of a certain operator. Finally, it is shown that this operator is a well defined contraction and thus has a unique fixed point. The result by Alber has also been improved upon by Tang and Xin \cite{TangXin2009} who prove the existence of a rotational solutions to the steady Euler equations which is the perturbation of a more general class of vector fields. This means that their result does not rely on the existence of a base flow that solves the steady Euler equations. Molinet \cite{Molinet1999} extended Alber's method to nonsmooth domains and compressible flows. However, he considers domains with a boundary consisting of smooth parts which meet at an angle smaller than $2\pi/7$ and any integer multiple of the angle may not equal $\pi$. Finally, we mention the result by Buffoni and Wahlén \cite{Buffoni2019}. They prove the existence of rotational solutions to the steady Euler equations in an unbounded domain of the form $(0,L)\times\mathbb{R}^2$, where the flows are periodic in both the unbounded directions. This is not an exhaustive list of the results and more can be found in the references within the cited works.

In this paper we prove existence of a rotational solutions to \eqref{eq:Euler1a}--\eqref{eq:Eulerbdry} in a simply connected domain $\Omega$, which has a boundary consisting of three $C^4$ parts that meet at a right angle.  We denote these three parts by $\partial\Omega_+$, $\partial\Omega_-$ and $\partial\Omega_0$, where the subscript denotes the sign of $\phi$ or that $\phi\equiv 0$. We also impose the restriction that $\overline{\partial\Omega}_-\cap\overline{\partial\Omega_+}=\emptyset$. This domain clearly is different from the ones in the results cited above. The easiest example of such a domain is one where the boundary is a right circular cylinder. Then the mantle is $\partial\Omega_0$ and the bases are $\partial\Omega_\pm$. In this case the solution can describe a rotational flow through a straight pipe. However many other domains are allowed. For example, one can allow a domain of the form depicted in figure \ref{Fig1}, where $\partial\Omega_\pm$ are flat. This can describe flow through curved pipes. One can also allow the surfaces $\partial\Omega_\pm$ not to be flat. However this demands an additional restriction on the curves $\partial\partial\Omega_\pm:=\overline{\partial\Omega_\pm}\cap \overline{\partial\Omega_0}$: they should be lines of curvature in any smoothness preserving extension of $\partial\Omega_\pm$ and $\partial\Omega_0$. Recall that a curve is a line of curvature in some surface $M$ if its tangent is a principal direction of $M$, i.e. the tangent vector is an eigenvector to the shape operator \cite[Section 3--2]{Carmo1976}. Note that if two surfaces intersect at a right angle, then the line of intersection is a line of curvature in one of the surfaces if and only if it is a line of curvature in the other. The condition that $\partial\partial\Omega_\pm$ are lines of curvature is in fact required even if $\partial\Omega_\pm$ is flat, but then it is automatically satisfied.

\begin{figure}
\begin{center}
\includegraphics[scale=1.2]{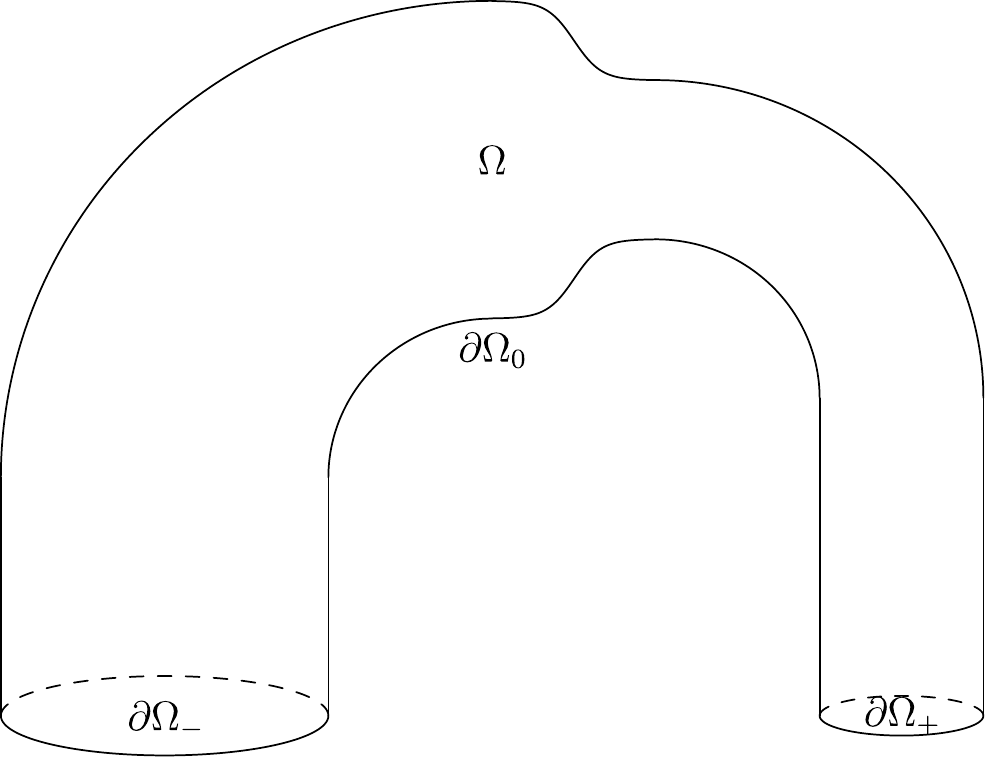}
\caption{A possible shape of $\Omega$}\label{Fig1}
\end{center}
\end{figure}

The proof itself is in overall structure similar to that in \cite{Alber1992}. It relies on a base flow which is perturbed by a fixed point of an operator to give us our solution. The idea behind this operator relies on rewriting the problem in the velocity-vorticity formulation, which means we replace \eqref{eq:Euler1a} with
\begin{equation}\label{eq:vel-vort}
(\bm{v}\cdot\nabla)\curl\bm{v}=(\curl\bm{v}\cdot\nabla)\bm{v}.
\end{equation}
More details about this can be found in \cite{Alber1992}. The operator itself is defined through first solving $(\bm{v}\cdot\nabla)\bm{f}=(\bm{f}\cdot\nabla)\bm{v}$ for some given $\bm{v}$, and then finding a velocity field with vorticity $\bm{f}$. The main difference, as compared to working in a smooth domain, lies in proving that the operator is a well defined contraction. This is done in two main steps, which we describe below in opposite order of the way the operator acts, because finding a velocity field with vorticity $\bm{f}$ puts some extra conditions on $\bm{f}$ that have to be incorporated in the problem of finding $\bm{f}$. In this order the first problem is finding a unique $\bm{w}$ which solves
\begin{equation}\label{prob:Div-Curl}
\begin{aligned}
\curl \bm{w}&=\bm{f} &&\inn \Omega,\\
\divv \bm{w}&=0 && \inn \Omega,\\
\bm{w}\cdot\bm{n}&=0 && \onn \partial\Omega,
\end{aligned}
\end{equation}
for a given $\bm{f}$. This will be referred to as the \textit{div-curl problem} $(\mathcal{DCP})$. For smooth domains this is a solved problem, but for the nonsmooth domains considered here results are more sparse. Zajaczkowski \cite{Zajaczkowski1988} proved an existence result but it requires $\bm{f}$ to satisfy certain compatibility conditions which are left implicit. We require these compatibility conditions formulated explictly to check that they are satisfied. Thus we prove an existence result in this paper with explicit compatibility conditions for $\bm{f}$. To do this we reduce the problem to three instances of the Poisson equation. This is a problem which has been studied with explicit compatibility conditions in various polyhedra with flat surfaces, see for example \cite{Grisvard1992, Grisvard2011}. However the author is unaware of any such results in the domains considered here. Thus, our results concerning this problem might be of independent interest. This is also where the condition that $\partial\partial\Omega_\pm$ is a curvature line shows up, which seems to be a novel observation. Possibly this is because the known results are focused on domains with flat surfaces where the condition is always satisfied. 

The second problem is, for given $\bm{v}$ and $\bm{f}_0$,  to find a unique $\bm{f}$ satisfying
\begin{equation}\label{prob:Transport}
\begin{aligned}
&(\bm{v}\cdot\nabla)\bm{f}=(\bm{f}\cdot \nabla )\bm{v} &&\inn \Omega,\\
&\bm{f}=\bm{f}_0 &&\onn \partial\Omega_-,
\end{aligned}
\end{equation}
together with the compatibility conditions mentioned above. We will refer to this as the \textit{transport problem} $(\mathcal{TP})$. The way $\bm{f}_0$ is chosen is the key to get a solution satisfying the compatibility conditions. Moreover, as long as $\bm{f}_0$ is non-trivial, we end up with a solution with nonvanishing vorticity. This problem is also solved differently than  in \cite{Alber1992}. Not so much because of the domains' geometry, but because we work with a slightly lower regularity, since it simplifies the compatibility conditions. However this means that we cannot use the same method as in \cite{Alber1992} to find a solution. Additionally, we have to prove various estimates for these problems, to show that the operator we are working with is indeed a contraction.

The overall layout for this article is as follows. In the next section we describe the function spaces which we are working with. In section 3 we present the main result in a more precise fashion, and its proof, given that we can solve the two problems formulated above. In section 4 we show that $(\mathcal{DCP})$ has a unique solution and prove some estimates related to this problem, and in section 5 we show the corresponding results and estimates for $(\mathcal{TP})$. In section 6 we show the existence of solutions to the irrotational problem in a cylindrical domain which satisfies the conditions we put on our base flow. This shows that all the assumptions of our main result can be fulfilled. Interestingly, this existence result seems to be missing in the literature, and could potentially be of independent interest.
 
\section{Function Spaces}
For a function $f:X\to Y$, where $X$ and $Y$ are Banach spaces, we use the notation
\[
\overline{f}:=\sup_{x\in X} \Vert f(x)\Vert_Y,\qquad \underline{f}:=\inf_{x\in X} \Vert f(x)\Vert_Y.
\]
Morover, we will make frequent use of the notation $A\lesssim B$ by which we mean that there exists a constant $C$ (independent of $A$ and $B$) such that $A\leq CB$.

Let $U$ be an open subset of $\mathbb{R}^n$ and let $X$ be a Banach space. We let $\mathcal{S}(U)$ denote the Schwartz space on $U$ and let $\mathcal{S}'(U;X):=\mathcal{L}(\mathcal{S}(U);X)$ be the space of vector valued tempered distributions (for a more comprehensive treatment of these spaces see e.g. \cite{Amann1997,Amann2000}). This allows us to define the function spaces we will mainly be working with in this paper.
\begin{definition}
Let $X$ be a Banach space and let $s\in \mathbb{R}$. We define $H^s(\mathbb{R}^n;X)$ as the Banach space of $f\in \mathcal{S}'(\mathbb{R}^n;X)$ such that
\[
\int_{\mathbb{R}^n}(1+\vert \bm\xi \vert^2)^{s}\Vert \mathcal{F}[f](\bm\xi)\Vert_X^2d\bm{\xi}<\infty,
\]
where $\mathcal{F}$ denotes the Fourier transform,
equipped with the norm
\[
\Vert f\Vert_{H^s(\mathbb{R}^n;X)}=\left(\int_{\mathbb{R}^n}(1+\vert \bm\xi \vert^2)^{s}\Vert \mathcal{F}[f](\bm\xi)\Vert_X^2d\bm{\xi}\right)^{1/2}.
\]
If $X$ is a Hilbert space, then $H^s(\mathbb{R}^n;X)$ becomes a Hilbert space with inner product
\[
\langle f,g\rangle_{H^s(\mathbb{R}^n;X)}=\int_{\mathbb{R}^n}(1+\vert \bm\xi \vert^2)^{s}\langle \mathcal{F}[f](\bm\xi),\mathcal{F}[g](\bm\xi)\rangle_Xd\bm{\xi}.
\]
\end{definition}
\begin{definition}
Let $U$ be an open subset of $\mathbb{R}^n$, X a Banach space and let $s\in \mathbb{R}$. We define $H^s(U;X)$ as the space of $f\in S'(U;X)$ such that there exists $g\in H^s(\mathbb{R}^n;X)$ with $f=g\vert_U$. The space is a Banach space with norm
\[
\Vert f\Vert_{H^s(U;X)}=\inf\{\Vert g\Vert_{H^s(\mathbb{R}^n;X)}: g\vert_U=f\}.
\]
\end{definition}
These spaces are extensions of the standard Sobolev spaces. If $s\in \mathbb{N}$ and $X=\mathbb{R}$ we have $H^s(U;X)=W^{s,2}(U)$. Moreover if $s\in \mathbb{R}$ and $X=\mathbb{R}$ they coicide with the Bessel potential spaces $H^s(U;X)=H^s(U)$. In appendix \ref{appendix} we include a technical result, which we need in section 5, about the spaces $H^s(U;X)$ in the special case when $X$ is another space of the same form.
\section{Main Result}
The aim of the paper is to prove the following result.
\begin{theorem}\label{Thm:main}
Let $\Omega\subset \mathbb{R}^3$ be a simply connected domain, whose boundary $\partial\Omega$ consists of three parts $C^4$ parts $\partial\Omega_-$, $\partial\Omega_+$ and $\partial\Omega_0$. Moreover, let $\overline{\partial\Omega}_+\cap\overline{\partial\Omega}_-=\emptyset$ and that partial $\partial\Omega_\pm$ meet at a right angle and that $\partial\partial\Omega_\pm$ are curvature lines. For any $1/2<\sigma<1$ let $\phi\in H^{3/2+\sigma}(\partial\Omega;\mathbb{R}^3)$ be a function satisfying \eqref{eq:phicompcond}. Assume that $(\bm{v}_0,p_0)\in H^{\sigma+2}(\Omega;\mathbb{R}^3)\times H^{\sigma+2}(\Omega;\mathbb{R})$ is a solution to equations \eqref{eq:Euler1a}--\eqref{eq:Eulerbdry} satisfying $\curl \bm{v}_0=0$ in $\Omega$ and $\underline{\bm{v}_0}>0$. Moreover, if the integral curves of $\bm{v}_0$ are of finite length and none of the integral curves of $\bm{v}_0$ are closed, then there exist constants
\begin{alignat*}{2}
\hat{\gamma}&>0,\\
K_i&>0, &&\qquad i\in\{1,2,3\},
\end{alignat*}
with the following properties:

(i) Let $g\in H^{\sigma+2}(\partial\Omega_-;\mathbb{R})$, $h\in H^{\sigma+1}(\partial\Omega_-;\mathbb{R})$ be functions that satisfy
\begin{equation}\label{eq:ghcond}
h\vert_{\partial\partial\Omega_-}=\nabla_T g\vert_{\partial\partial\Omega_-}=0,
\end{equation}
and
\begin{equation}\label{eq:smallbdrydata}
\Vert h\Vert_{H^{\sigma+1}(\partial\Omega_-;\mathbb{R})}+\Vert \nabla_Tg\Vert_{H^{\sigma+1}(\partial\Omega_-;\mathbb{R}^3)}\leq K_1,
\end{equation}
where $\nabla_Tg$ denotes the tangential derivative of $g$.
Then there exists a solution $(\bm{v},p)\in H^{\sigma+2}(\Omega;\mathbb{R}^3)\times H^{\sigma+2}(\Omega;\mathbb{R})$ to \eqref{eq:Euler1a}--\eqref{eq:Eulerbdry}, which also satisfies
\begin{alignat}{2}
\bm{n}\cdot\curl \bm{v}&=h &&\qquad \onn \partial\Omega_-,\label{eq:bdry1}\\
\frac{1}{2}\vert \bm{v}\vert^2+p&=g+\frac{1}{2}\vert \bm{v}_0\vert^2+p_0 &&\qquad \onn \partial\Omega_-.\label{eq:bdry2}
\end{alignat}

(ii) The velocity field of this solution, $\bm{v}$, satisfies 
\begin{equation}\label{eq:vclose}
\Vert \bm{v}-\bm{v}_0\Vert_{H^{\sigma+2}(\Omega;\mathbb{R}^3)}\leq \hat{\gamma}
\end{equation}
and $(\bm{v},p)$ is the only solution to \eqref{eq:Euler1a}--\eqref{eq:Eulerbdry}, \eqref{eq:bdry1} and \eqref{eq:bdry2} in $H^{\sigma+2}(\Omega;\mathbb{R}^3)\times H^{\sigma+2}(\Omega;\mathbb{R})$ with $\bm{v}$ satisfying \eqref{eq:vclose}.

(iii) If $(g^{(1)}, h^{(1)})$ and $(g^{(2)},h^{(2)})$ are two sets of boundary data on $\partial\Omega_-$ both satisfying \eqref{eq:smallbdrydata} with corresponding solutions $(\bm{v}^{(1)},p^{(1)})$ and $(\bm{v}^{(2)},p^{(2)})$ to \eqref{eq:Euler1a}--\eqref{eq:Eulerbdry}, \eqref{eq:bdry1} and \eqref{eq:bdry2}, both satisfying \eqref{eq:vclose}, then
\begin{equation}\label{eq:contdepvelocity}
\Vert \bm{v}^{(1)}-\bm{v}^{(2)}\Vert_{H^1(\Omega;\mathbb{R}^3)}\leq K_2\left(\Vert h^{(1)}-h^{(2)}\Vert_{L^2(\partial\Omega_-;\mathbb{R})}+\Vert \nabla_T(g^{(1)}-g^{(2)})\Vert_{L^2(\partial\Omega_-;\mathbb{R}^3)}\right)
\end{equation}
and
\begin{equation}\label{eq:contdeppressure}
\begin{aligned}
\Vert p^{(1)}-p^{(2)}\Vert_{H^1(\Omega;\mathbb{R})}&\leq K_3\left(\Vert h^{(1)}-h^{(2)}\Vert_{L^2(\partial\Omega_-;\mathbb{R})}+\Vert \nabla_T(g^{(1)}-g^{(2)})\Vert_{L^2(\partial\Omega_-;\mathbb{R}^3)}\right.\\
&\qquad\qquad+\left.\Vert g^{(1)}-g^{(2)}\Vert_{L^2(\partial\Omega_-;\mathbb{R})}\right).
\end{aligned}
\end{equation}
\end{theorem}
\begin{remark} We will refer to the task of finding a solution $(\bm{v},p)\in H^{s+1}(\Omega;\mathbb{R}^3)\times H^{s+1}(\Omega;\mathbb{R})$ to \eqref{eq:Euler1a}--\eqref{eq:Eulerbdry}, \eqref{eq:bdry1} and \eqref{eq:bdry2} as the problem $(\mathcal{P})$.
\end{remark}
\begin{remark} We show the existence of an irrotational solution $(\bm{v}_0,p)$ satisfying the requirements given in Theorem \ref{Thm:main} in section \ref{sec:Irrotational} for the special case when $\Omega$ is a cylinder. The problem of finding such a solution in a general domain is beyond the scope of this paper.
\end{remark}
\begin{remark} $\sigma$ will denote a fixed real number between $1/2$ and $1$ throughout the rest of the paper.
\end{remark}
\subsection{Proof of the Main Result}
The idea is to define an operator, $B$, in such a way that $(\bm{v},p)$ is a solution to $(\mathcal{P})$ if and only if $\bm{u}$ is a fixed point of $B$, where $\bm{u}=\bm{v}-\bm{v}_0$. Then we show that $B$ is well defined and a contraction on a sufficiently small neighbourhood of the origin. This gives us a unique fixed point, and hence the desired solution, by the Banach fixed point theorem. 

To define $B$ we start by defining $V$ as the space of functions $\bm{u}\in H^{\sigma+2}(\Omega;\mathbb{R}^3)$ which satisfy
\begin{equation*}
\begin{aligned}
\divv \bm{u}&=0 &&\inn \Omega,\\
\bm{u}\cdot \bm{n}&=0 &&\onn \partial\Omega.
\end{aligned}
\end{equation*}
Then by $V_\gamma$ we denote the closed ball of radius $\gamma$ in $V$, i.e. the functions $\bm{u}\in V$ satisfying $\Vert \bm{u}\Vert_{H^{\sigma+2}(\Omega;\mathbb{R}^3)}\leq \gamma$.
Now $B:V_\gamma\to V$ is defined as 
\[
B(\bm{u})=\bm{w}
\]
where $\bm{w}$ is given by
\[
\curl\bm{w}=\bm{f}
\]
and $\bm{f}$ is the solution to $(\mathcal{TP})$ where $\bm{v}=\bm{v}_0+\bm{u}$ and $\bm{f}_0$ is defined by
\begin{alignat}{2}
&\bm{f}_0\cdot \bm{n}=h &&\qquad \onn \partial\Omega_-,\label{eq:f0defn}\\
&{\bm{f}_0}_T=\frac{h}{\bm{v}\cdot\bm{n}}\bm{v}_T-\frac{1}{\bm{v}\cdot\bm{n}}\bm{n}\times\nabla_Tg &&\qquad \onn \partial\Omega_-.\label{eq:f0deft}
\end{alignat}
To show that $B$ is well defined we have to find a unique $\bm{f}\in H^{\sigma+1}(\Omega;\mathbb{R}^3)$ that solves $(\mathcal{TP})$, that is a solution to
\[
\begin{aligned}
&(\bm{v}\cdot\nabla)\bm{f}=(\bm{f}\cdot \nabla )\bm{v} &&\inn \Omega,\\
&\bm{f}=\bm{f}_0 &&\onn \partial\Omega_-,
\end{aligned}
\]
and find a unique $\bm{w}\in H^{\sigma+2}(\Omega)$ which solves $(\mathcal{DCP})$, i.e. a solution to
\[
\begin{aligned}
\curl \bm{w}&=\bm{f} &&\inn \Omega,\\
\divv \bm{w}&=0 && \inn \Omega,\\
\bm{w}\cdot\bm{n}&=0 && \onn \partial\Omega,
\end{aligned}
\]
As mentioned in the introduction solving $(\mathcal{DCP})$ requires some additional conditions to be imposed on $\bm{f}$, which further complicates $(\mathcal{TP})$. Finding a solution to $(\mathcal{TP})$ that satisfies these conditions can be achieved by imposing additional conditions on $\bm{f}_0$ (and by extension $h$ and $g$).

The solutions to these problems and the conditions discussed above are given by the two following theorems.
\begin{theorem}\label{Theorem:TP} Let $\bm{v}_0$ satisfy the hypothesis of Theorem \ref{Thm:main}. There exists a constant $\gamma_0$ such that if $\gamma\leq \gamma_0$ and $\bm{u}\in V_\gamma$, then a unique solution $\bm{f}\in H^{\sigma+1}(\Omega;\mathbb{R}^3)$ to $(\mathcal{TP})$ exists and satisfies $\divv \bm{f}=0$. Moreover, the solution satisfies the estimates
\begin{subequations}
\begin{align}
\Vert \bm{f}\Vert_{H^{\sigma+1}(\Omega;\mathbb{R}^3)}&\lesssim\Vert \bm{f}_0\Vert_{H^{\sigma+1}(\partial\Omega_-;\mathbb{R}^3)},\label{tpest1}\\
\Vert \bm{f}\Vert_{L^2(\Omega;\mathbb{R}^3)}&\lesssim\Vert \bm{f}_0\Vert_{L^2(\partial\Omega_-;\mathbb{R}^3)}.\label{tpest2}
\end{align}
Additionally we have
\begin{equation}\label{tpest3}
\Vert \bm{f}^{(2)}-\bm{f}^{(1)}\Vert_{L^2(\Omega;\mathbb{R}^3)}\lesssim\Vert \bm{f}_0\Vert_{H^{\sigma+1}(\partial\Omega_-;\mathbb{R}^3)}\Vert \bm{v}^{(2)}-\bm{v}^{(1)}\Vert_{H^1(\Omega;\mathbb{R}^3)},
\end{equation}
for two different solutions $\bm{f}^{(1)}$ and $\bm{f}^{(2)}$ corresponding to $\bm{v}^{(1)}$ and $\bm{v}^{(2)}$ respectively.
\end{subequations} Finally, given that $\bm{f}_0\vert_{\partial\partial\Omega_-}\equiv 0$ then $\bm{f}\vert_{\partial\partial\Omega_+}\equiv 0$.
\end{theorem}
The proof of Theorem \ref{Theorem:TP} is given in Section \ref{Sec:TP}.
\begin{theorem}\label{Theorem:DCP} Given $\bm{f}$ in $H^{\sigma+1}(\Omega;\mathbb{R}^3)$ which satisfies $\divv \bm{f}=0$ and $\bm{f}\vert_{\partial\partial\Omega_\pm}\equiv 0$ the $(\mathcal{DCP})$ has a unique solution $\bm{w}\in H^{\sigma+2}(\Omega;\mathbb{R}^3)$ which satisfies
\begin{subequations}
\begin{align}
\Vert \bm{w}\Vert_{H^{\sigma+2}(\Omega;\mathbb{R}^3)}\lesssim \Vert \bm{f}\Vert_{H^{\sigma}(\Omega;\mathbb{R}^3)},\label{dcpest1}\\
\Vert \bm{w}\Vert_{H^1(\Omega;\mathbb{R}^3)}\lesssim \Vert \bm{f}\Vert_{L^2(\Omega;\mathbb{R}^3)}.\label{dcpest2}
\end{align}
\end{subequations}
\end{theorem}
The proof of Theorem \ref{Theorem:DCP} is given in section \ref{Sec:DCP}. By combining the two previous theorems we get the following result.
\begin{lemma} Let $g$, $h$ and $\bm{v}_0$ be given as in in Theorem \ref{Thm:main}. Then the operator $B=B[g,h,\bm{v}_0]:V_\gamma \to V$ is well defined. Moreover we have the following:

(i) For every $\gamma\leq \gamma_0$ there exists a constant $K_1$ such that $B$ maps $V_\gamma$ into itself and $B:V_\gamma\subset H^1(\Omega;\mathbb{R}^3)\to H^1(\Omega;\mathbb{R}^3)$ is a contraction.

(ii) $B$ has a unique fixed point in $V_\gamma$.
\end{lemma}
\begin{proof} That $B:V_\gamma\to V$ is well defined follows directly from the theorems. Since $g\in H^{\sigma+2}(\partial\Omega_-;\mathbb{R}^3)$ and $h\in H^{\sigma+1}(\partial\Omega_-;\mathbb{R}^3)$ satisfying $\eqref{eq:ghcond}$ equations \eqref{eq:f0defn} and \eqref{eq:f0deft} give $\bm{f}\in H^{\sigma+1}(\partial\Omega_-;\mathbb{R}^3)$, which satisfies $\bm{f}_0\vert_{\partial\partial\Omega_-}\equiv 0$. Thus by Theorem \ref{Theorem:TP} we get a unique function $\bm{f}$ that satisfies the hypothesis of Theorem \ref{Theorem:DCP}, which in turn gives us a unique function $\bm{w}\in V$ that satisfies $\curl \bm{w}=\bm{f}$.

To prove part $(i)$ of the lemma we note that combining the inequalities \eqref{tpest1} and \eqref{dcpest1} gives us that
\[
\Vert B[g,h,v_0](\bm{u})\Vert_{H^{\sigma+2}(\Omega;\mathbb{R}^3)}\lesssim\Vert \bm{f}_0\Vert_{H^{\sigma+1}(\partial\Omega_-;\mathbb{R}^3)}.
\]
Furthermore, by the definition of $\bm{f}_0$, we know that
\[
\bm{f}_0=(\bm{f}_0\cdot\bm{n})\bm{n}+{\bm{f}_0}_T=h\bm{n}+\frac{h}{\bm{v}\cdot\bm{n}}\bm{v}_T-\frac{1}{\bm{v}\cdot\bm{n}}\bm{n}\times\nabla_Tg.
\]
Thus 
\[
\Vert\bm{f}_0 \Vert_{H^{\sigma+1}(\partial\Omega_-;\mathbb{R}^3)}\lesssim
\Vert h\Vert_{H^{\sigma+1}(\partial\Omega_-;\mathbb{R})}+\Vert \nabla_T g\Vert_{H^{\sigma+1}(\partial\Omega_-;\mathbb{R}^3)},
\]
which means there exists a constant $C_1$ such that
\begin{equation}\label{eq:opBest1}
\Vert B[g,h,v_0](\bm{u})\Vert_{H^{\sigma+2}(\Omega;\mathbb{R}^3)}\leq C_1(\Vert h\Vert_{H^{\sigma+1}(\partial\Omega_-;\mathbb{R})}+\Vert \nabla_T g\Vert_{H^{\sigma+1}(\partial\Omega_-;\mathbb{R}^3)}).
\end{equation}
If we choose $K_1$ in small enough for $C_1K_1\leq \gamma$, then we get that $B$ maps $V_\gamma$ into itself.

Since the $(\mathcal{DCP})$ is a linear problem the estimate in \eqref{dcpest2} gives us
\[
\Vert \bm{w}^{(2)}-\bm{w}^{(1)} \Vert_{H^1(\Omega;\mathbb{R}^3)}\lesssim \Vert \bm{f}^{(2)}-\bm{f}^{(1)}\Vert_{L^2(\Omega;\mathbb{R}^3)},
\]
for two different solutions $\bm{w}^{(1)}$ and $\bm{w}^{(2)}$ corresponding to two different functions $\bm{f}^{(1)}$ and $\bm{f}^{(2)}$.
Together with the estimate \eqref{tpest3} we get, for two different $\bm{u}^{(1)}, \bm{u}^{(2)}\in V_\gamma$,
\begin{equation*}
\Vert B[g,h,\bm{v}_0](\bm{u}^{(2)})-B[g,h,\bm{v}_0](\bm{u}^{(1)})\Vert_{H^1(\Omega;\mathbb{R}^3)}\lesssim
\Vert \bm{f}_0\Vert_{H^s(\partial\Omega_-;\mathbb{R}^3)}\Vert \bm{u}^{(2)}-\bm{u}^{(1)}\Vert_{H^1(\Omega;\mathbb{R}^3)}
\end{equation*}
because $\bm{v}^{(2)}-\bm{v}^{(1)}=\bm{u}^{(2)}-\bm{u}^{(1)}$. Using our estimate for $\Vert\bm{f}_0 \Vert_{H^{\sigma+1}(\partial\Omega_-;\mathbb{R}^3)}$ gives us that there exists a constant $C_2$ such that
\begin{equation}\label{eq:opBest2}
\begin{aligned}
\Vert B[g,h,\bm{v}_0](\bm{u}^{(2)})-B[g,h,\bm{v}_0](\bm{u}^{(1)})\Vert_{H^1(\Omega;\mathbb{R}^3)}&\leq
\\ 
C_2(\Vert h\Vert_{H^{\sigma+1}(\partial\Omega_-;\mathbb{R})}&+\Vert \nabla_T g\Vert_{H^{\sigma+1}(\partial\Omega_-;\mathbb{R}^3)})\Vert \bm{u}^{(2)}-\bm{u}^{(1)}\Vert_{H^1(\Omega;\mathbb{R}^3)}.
\end{aligned}
\end{equation}
From this it follows that if we in addition choose $K_1$ so that $C_2K_1<1$ then $B:V_\gamma\subset H^1(\Omega;\mathbb{R}^3)\to H^1(\Omega;\mathbb{R}^3)$ is a contraction.

For part $(ii)$ we use that $B:V_\gamma\subset H^1(\Omega;\mathbb{R}^3)\to H^1(\Omega;\mathbb{R}^3)$ is a contraction. By the Banach fixed-point theorem, iterating $B$ will give us a sequence which converges to a unique fixed point. What remains to show is that $V_\gamma$ is closed in $H^1$ ensuring that our fixed point lies in $V_\gamma$. This can be done by employing the same technique used in \cite{Alber1992}. Any sequence in $V_\gamma$ has a weakly convergent subsequence in $H^{\sigma+2}(\Omega;\mathbb{R}^3)$. The subsequence clearly has the same weak limit in $H^1(\Omega;\mathbb{R}^3)$. If it also is convergent in $H^1(\Omega;\mathbb{R}^3)$, then the weak limit and the limit are the same. Thus the limit is a function in $H^{\sigma+2}(\Omega;\mathbb{R}^3)$. The other conditions follow by continuity of the $\divv$ and trace operators on $H^1(\Omega;\mathbb{R}^3)$.
\end{proof}
What remains to prove is that $(\bm{v},p)$ is a solution to $(\mathcal{P})$ if and only if $\bm{u}=\bm{v}-\bm{v}_0$ is a fixed point of $B$, and the estimates \eqref{eq:contdepvelocity} and \eqref{eq:contdeppressure}. To prove the former we need the following lemma, which is very reminiscent of Lemma 2.1 in \cite{Alber1992} and can be proven in the same way.

\begin{lemma}\label{Lemma:IntegralCurves} Let $\bm{v}_0$ satisfy the hypothesis of Theorem \ref{Thm:main}. Then there exist constants  $C$ and $\gamma_0$ such that the following three properties hold:

(i) For any $\bm{v}=\bm{v}_0+\bm{u}$, where $\bm{u}\in V_{\gamma_0}$
\[
\underline{\bm{v}}>\underline{\bm{v}_0}-C\gamma_0>0.
\]

(ii) Let $0<\gamma\leq \gamma_0$. Then no vector field $\bm{v}\in \bm{v}_0+V_\gamma$ has closed integral curves. Furthermore if $L_\gamma$ denotes the least upper bound of all integral curves to all vector fields in $\bm{v}_0+V_\gamma$ then $L_\gamma<\infty$ and
\[
\lim_{\gamma\to 0}L_\gamma=L_0.
\] 

(iii) If an integral curve to a vector field in $\bm{v}_0+V_\gamma$ is tangential to the boundary at any point then it is completely contained in the boundary. 
\end{lemma}

\begin{remark} The constant $\gamma_0=\gamma_0(\bm{v}_0)$ in this lemma is the same as the one in Theorem \ref{Theorem:TP}.
\end{remark}

The consequence of this lemma is that the integral curves of any vector field $\bm{v}=\bm{v}_0+\bm{u}$ covers all of $\Omega$ and that they all intersect $\partial\Omega_-$ in exactly one point and $\partial\Omega_+$ in exactly one point. Given any point in $\Omega$, we can follow the integral curve of $\bm{v}$ from that point. It will not reach a stagnation point by part $(i)$ neither will it return to the same point by part $(ii)$. Furthermore by part $(ii)$ the integral curve has finite length so it will eventually reach $\partial\Omega_+$. The same is true if we follow the integral curve backwards except that we eventually reach $\partial\Omega_-$ instead.

With the help of this lemma we can prove that $(\bm{v},p)$ is a solution to $(\mathcal{P})$ if and only if $\bm{u}=\bm{v}-\bm{v_0}$ is a fixed point of $B$ using the same method used to prove Lemma 2.6 in \cite{Alber1992} and the estimates in \eqref{eq:contdepvelocity} and \eqref{eq:contdeppressure} can be proved in a similar manner to how the corresponding estimates are shown in the proof of Theorem 1.1 in the same article.

\section{Div-Curl Problem}\label{Sec:DCP}
This entire section is the proof of Theorem \ref{Theorem:DCP}.
\subsection{Formulation of the Problem}
The aim is to find a solution, $\bm{w}\in H^{\sigma+2}(\Omega;\mathbb{R}^3)$, to the $(\mathcal{DCP})$ given $\bm{f}\in H^{\sigma+1}(\Omega;\mathbb{R}^3)$, the solution to $(\mathcal{TP})$ from Theorem \ref{Theorem:TP}. For this purpose we assume that a solution is already known and introduce a vector potential $\bm{u}$. We can do this by applying Theorem 3.17 in \cite{AmroucheBernardiDaugeEtAl1998}. The vector potential satisfies $\curl \bm{u}=-\bm{w}$, $\divv \bm{u}=0$ in $\Omega$, and $\bm{u}\times\bm{n}=0$ at $\partial\Omega$. This means that $\Delta\bm{u}=\bm{f}$ and, because $\divv \bm{f}=0$, the condition $\divv \bm{u}=0$ in $\Omega$ is equivalent to the same on $\partial\Omega$. Hence, instead of the $(\mathcal{DCP})$ in its original form, we consider the problem
\begin{subequations}\label{orig:prob}
\begin{align}
\Delta \bm{u}&=\bm{f} &&\inn \Omega, \label{orig:eq}
\\
\nabla\cdot \bm{u}&=0 && \onn \partial\Omega, \label{orig:BC1}
\\
\bm{u}\cdot\bm{\tau}^{(i)}_j&=0 &&\onn \partial\Omega_i, \qquad \text{ for } i\in\{+,-,0\}\text{ and } j\in\{1,2\}.\label{orig:BC2}
\end{align}
\end{subequations}
To obtain $\bm{w}$ of the required regularity we seek a solution $\bm{u}\in H^{\sigma+3}(\Omega;\mathbb{R}^3)$. Here $\bm{\tau}^{(i)}_1$ and $\bm{\tau}^{(i)}_2$ denote two vector valued functions defined on $\partial\Omega_i$, which are tangent to $\partial\Omega_i$ and linearly independent at every point.
\subsection{Auxiliary Results}
We first consider some auxiliary problems in the domain $D\subset\mathbb{R}^3$ given by
\[
D=\{(x_1,x_2,x_3): 0<x_1<d, 0<x_2<d\}.
\]
The boundary consists of four open faces $\Gamma_i$, $i\in\{1,2,3,4\}$ given by
\[
\begin{aligned}
\Gamma_1&=\{(x_1,x_2,x_3): x_1=0, 0<x_2<d\}\\
\Gamma_2&=\{(x_1,x_2,x_3): 0<x_1<d, x_2=0\}\\
\Gamma_3&=\{(x_1,x_2,x_3): x_1=d, 0<x_2<d\}\\
\Gamma_4&=\{(x_1,x_2,x_3): 0<x_1<d, x_2=d\},
\end{aligned}
\]
where we set $\Gamma_{i+4}=\Gamma_i$, and the edges $S_i:=\bar\Gamma_i\cap\bar\Gamma_{i+1}$ for $i\in\{1,2,3,4\}$. In this domain we are interested in Poisson's equation with either Dirichlet or Neumann boundary conditions on the different faces. For this reason we define $B_i$ to denote either the identity operator or the derivative in the normal direction of $\Gamma_i$. 
With this notation we can express the problem as
\begin{equation}\label{DirNeuNonhom}
\begin{aligned}
\Delta u&=f&&\qquad \inn D,\\
B_iu&=\zeta_i&&\qquad \onn \Gamma_i,\;i=1,2,3,4.
\end{aligned}
\end{equation}
To formulate the result about the solvability of this problem we introduce
\[
J_i=\begin{cases} 1 \text{ if $B_i$ represents the identity operator,}\\
0 \text{ if $B_i$ represents the derivative in the normal direction of $\Gamma_i$,}
\end{cases}
\]
and $K_i=J_i+J_{i+1}$. The result is given in the following proposition.
\begin{proposition}\label{Prop:DirNeu}Let $m\in \mathbb{Z}$ with $-1\leq m\leq 1$. If $\sum_{i=1}^4 J_i>0$, and $\zeta_i\in H^{m+\sigma+J_i+\frac{1}{2}}(\Gamma_i;\mathbb{R})$, $i\in \{1,2,3,4\}$, and $f\in H^{m+\sigma}(D;\mathbb{R})$ satisfy
\begin{align}
B_{i+1}\zeta_i&=B_i\zeta_{i+1} && \onn S_i,\, i\in\{1,2,3,4\},\text{ if }m+K_i\geq 1, \label{compcondb}
\\
f&=0 && \onn S_i,\, i\in\{1,2,3,4\}, \text{ if }K_i=2\text{ and }m=1.\label{compcondc}
\end{align} 
Then the problem \eqref{DirNeuNonhom} has a unique solution $u\in H^{m+\sigma+2}(D)$. Moreover the solution satisfies the estimate
\[
\Vert u\Vert_{H^{m+\sigma+2}(D;\mathbb{R})}\lesssim \left(\Vert f\Vert_{H^{m+\sigma}(D;\mathbb{R})}+\sum_{i=1}^4\Vert \zeta_i \Vert_{H^{m+\sigma+J_i+\frac{1}{2}}(\Gamma_i;\mathbb{R})}\right).
\]
\end{proposition}

The proof of this proposition relies on a similar result concerning the corresponding homogeneous problems. This result is given by Theorem 2.5.11 in \cite{Grisvard1992} and yields $H^2(D;\mathbb{R})$ solutions to the homogeneous problems. By the extention result given by Theorem 1.4.3.1 in \cite{Grisvard2011} and by standard interpolation results see e.g. Section 2.4.1 in \cite{Triebel1978} we can extend this to Sobolev spaces of fractional regularity. Finally with the trace results in Theorem 6.9 and Corollary 6.10 in \cite{BernadiDaugeMaday2007} we can extend this to the nonhomogeneous results given above.

\subsection{Local Change of Variables}
The way we solve problem \eqref{orig:prob} is through a localization argument. $\Omega$ is covered by open neighbourhoods and with a partition of unity we replace the original problem with a local problem in each neighbourhood. Problems in the interior neighbourhoods and neighbourhoods that only intersect one of the boundary pieces can be treated in the same way as the problem for a smooth domain. Because of this we only treat the neighbourhoods that intersect an edge below.

To this end we turn our attention to one of the edges where $\partial\Omega_0$ and either $\partial\Omega_+$ or $\partial\Omega_-$ meet. To simplify notation below we relabel the sides that meet and call them $\partial\Omega_1$ and $\partial\Omega_2$. We also assume that a point on the edge is the origin in our coordinates $\bm{y}=(y_1,y_2,y_3)\in\mathbb{R}^3$. Moreover we assume that the unit vectors $\bm{e}_i$ are normal to $\partial\Omega_i$ for $i=1,2$ at the origin and $\bm{e}_3$ is tangent to the edge at the origin.

Let $\omega$ be some (sufficiently small) neighbourhood of the origin. We define a change of variables
\[
(x_1,x_2,x_3)=\bm{x}=\bm{\Psi}(\bm{y})
\]
which takes the origin to the origin, $\Omega\cap \omega$ to a neighbourhood of the origin, $\tilde{\omega}$ in $M=\{\bm{x}\in\mathbb{R}^3| 0<x_1,0<x_2\}$ and $\partial\Omega_i\cap \omega$ to $\tilde{\omega}\cap\partial M_i$, where $\partial M_i=\{\bm{x}\in \mathbb{R}^3|x_i=0\}$. Moreover let $\bm{\Psi}$ be defined in such a way that $\bm{\Psi}'(0)=\bm{I}$, and $\omega$ is chosen so small that $\det \bm{\Psi}'(\bm{y})>\frac{1}{2}$ for all $\bm{y}\in \omega$.

Let $\tilde{\rho}$ be a function which satisfies $\tilde{\rho}\equiv 1$ on $[-1/2,1/2]$ and $\tilde{\rho}\equiv 0$ outside $[-1,1]$. Set $\rho(\bm{y})=\tilde{\rho}(y_1)\tilde{\rho}(y_2)\tilde{\rho}(y_3)$ and define
\[
\bm{\Phi}(\bm{y})=\rho(\bm{y}/\epsilon)\bm{\Psi}(\bm{y})+(1-\rho(\bm{y}/\epsilon))\bm{y},
\]
for some $\epsilon$ to be chosen later, but small enough for $[-\epsilon,\epsilon]^3\subset \omega$, so $\bm{\Phi}$ can be extended to $\mathbb{R}^3$ by extending $\bm{\Psi}$ to $\mathbb{R}^3$. In the following we let $\nabla_{\bm{y}}=(\partial_{y_1},\partial_{y_2},\partial_{y_3})$, $\Delta_{\bm{y}}=\partial_{y_1}^2+\partial_{y_2}^2+\partial_{y_3}^2$ and the same with $\bm{y}$ replaced by $\bm{x}$.

Let $\eta$ be a smooth cutoff function, with support contained in $[-\epsilon/2,\epsilon/2]^3$, defined in such a way that $\eta\equiv 1$ in some open subset of $\omega$ containing the origin. If we have a regular enough solution $\bm{u}$ of \eqref{orig:prob} then $\bm{\tilde{u}}=\eta\bm{u}$ will satisfy
\begin{subequations}\label{cutoff:prob}
\begin{align}
\Delta_{\bm{y}} \bm{\tilde{u}}&=\bm{\tilde{f}}+2(\nabla_{\bm{y}}\eta\cdot \nabla_{\bm{y}})\bm{u}+ \bm{u}\Delta_{\bm{y}}\eta &&\inn \omega\cap\Omega, \label{cutoff:eq}
\\
\nabla_{\bm{y}}\cdot \bm{\tilde{u}}&=-\bm{u}\cdot\nabla_{\bm{y}}\eta &&\onn \omega\cap\partial\Omega_i, \qquad \text{ for } i\in\{1,2\}, \label{cutoff:BC1}
\\
\bm{\tilde{u}}\cdot\bm{\tau}^{(i)}_j&=0 &&\onn \omega\cap\partial\Omega_i, \qquad \text{ for } i\in\{1,2\}\text{ and } j\in\{1,2\},\label{cutoff:BC2}
\\
\bm{\tilde{u}}&=0 && \onn \partial\omega\setminus \partial\Omega_1\cup\partial\Omega_2.\label{cutoff:BC3}
\end{align}
\end{subequations}
for $\bm{\tilde{f}}=\bm{f}\eta$.

We define the two inward unit normal vector fields $\bm{n}_i$ on $\partial\Omega_i$, $i\in\{ 1, 2\}$ and then extend them to $\omega$ in such a way that $\bm{n}_1\cdot \bm{n}_2 = 0$. Note that this is possible since $\bm{n}_1\cdot \bm{n}_2 = 0$ at the edge by assumption. We also set $\bm{n}_3 = \bm{n}_1 \times \bm{n}_2$. This makes $\bm{n}_3$ tangent to $\partial\Omega_i$, $i\in\{1,2\}$, and in particular tangent to the edge where $\partial\Omega_1$ and $\partial\Omega_2$ meet.
Defining $\tilde{u}_i=\bm{\tilde{u}}\cdot\bm{n}_i$, $i\in\{1,2,3\}$, we can rewrite \eqref{cutoff:prob} as three coupled problems for the components $\tilde{u}_i$, $i\in\{1,2,3\}$. If we extend everything by zero to $T=\bm{\Phi}^{-1}(D)$, with boundary composed of $R_i=\bm{\Phi}^{-1}(\Gamma_i)$, $i\in\{1,2,3,4\}$, and set
\begin{align*}
g_k&=(\bm{\tilde{f}}+2(\nabla_{\bm{y}}\eta\cdot\nabla_{\bm{y}})\bm{u}+ \bm{u}\Delta_{\bm{y}}\eta)\cdot \bm{n}_k+2\sum_i\partial_{y_i}\bm{\tilde{u}}\cdot\partial_{y_i}\bm{n}_k+\bm{\tilde{u}}\cdot\Delta_{\bm{y}} \bm{n}_k,
\\
h_i&=-\vert\nabla_{\bm{y}}{\Phi}_i\vert(\tilde{u}_i\nabla_{\bm{y}} \cdot \bm{n}_i+\bm{u}\cdot\nabla_{\bm{y}}\eta),
\end{align*}
then we get

\begin{subequations}\label{loc1:prob}
\begin{align}
\Delta_{\bm{y}} \tilde{u}_1&=g_1&&\inn T, \label{loc1:eq}
\\
\nabla_{\bm{y}}{\Phi}_1\cdot\nabla_{\bm{y}} \tilde{u}_1&=h_1 &&\onn R_1 \label{loc1:BC1}
\\
\tilde{u}_1&=0 &&\onn R_i,\qquad\text{for } i\in\{2,3,4\},\label{loc1:BC2} 
\end{align}
\end{subequations}
\begin{subequations}\label{loc2:prob}
\begin{align}
\Delta_{\bm{y}}\tilde{u}_2&=g_2 &&\inn T, \label{loc2:eq}
\\
\tilde{u}_2&=0 &&\onn R_i,\qquad\text{for } i\in\{1,3,4\},\label{loc2:BC1}
\\
\nabla_{\bm{y}}{\Phi}_2\cdot \nabla_{\bm{y}} \tilde{u}_2&=h_2 &&\onn R_2, \label{loc2:BC2}
\end{align}
\end{subequations}
\begin{subequations}\label{loc3:prob}
\begin{align}
\Delta_{\bm{y}} \tilde{u}_3&=g_3 &&\inn T, \label{loc3:eq}
\\
\tilde{u}_3&=0 &&\onn R_i, \qquad\text{for } i\in\{1,2,3,4\}.\label{loc3:BC1}
\end{align}
\end{subequations}
Here we have used the fact that $\nabla_{\bm{y}}\Phi_i$ is parallel to $\bm{n}_i$ on the boundary $\partial\Omega_i$ for $i\in\{1,2\}$. We apply our change of variables $\bm{\Phi}$ to \eqref{loc1:prob}--\eqref{loc3:prob}. Setting 
\begin{align}
v_k&=\tilde{u}_k\circ\bm{\Phi}^{-1}&&\text{ for } k\in\{1,2,3\},\\ 
g_k^{(x)}&=g_k\circ\bm{\Phi}^{-1}&&\text{ for } k\in\{1,2,3\},\label{eq:gkx}\\
h_i^{(x)}&=h_i\circ\bm{\Phi}^{-1}&&\text{ for }  i\in\{1,2\},\label{eq:hkx}
\end{align} gives
\begin{subequations}\label{trans1:prob}
\begin{align}
\Delta_{\bm{x}} v_1+(\Delta_{\bm{y}}-\Delta_{\bm{x}}) v_1&=g_1^{(x)}&&\inn D, \label{trans1:eq}
\\
\partial_{x_1} v_1+(\nabla_{\bm{y}}\Phi_1\cdot\nabla_{\bm{y}} -\partial_{x_1}) v_1&=h_1^{(x)} &&\onn \Gamma_1, \label{trans1:BC1}
\\
v_1&=0 &&\onn \Gamma_i,\qquad\text{for } i\in\{2,3,4\},\label{trans1:BC2} 
\end{align}
\end{subequations}
\begin{subequations}\label{trans2:prob}
\begin{align}
\Delta_{\bm{x}} v_2 +(\Delta_{\bm{y}}-\Delta_{\bm{x}})v_2&=g_2^{(x)}&&\inn D, \label{trans2:eq}
\\
v_2&=0 &&\onn \Gamma_i,\qquad\text{for } i\in\{1,3,4\},\label{trans2:BC1}
\\
\partial_{x_2} v_2+ (\nabla_{\bm{y}}\Phi_2\cdot \nabla_{\bm{y}}-\partial_{x_2}) v_2&=h_2^{(x)} &&\onn \Gamma_2, \label{trans2:BC2}
\end{align}
\end{subequations}
\begin{subequations}\label{trans3:prob}
\begin{align}
\Delta_{\bm{x}} v_3+(\Delta_{\bm{y}}-\Delta_{\bm{x}}) v_3&=g_3^{(x)}&&\inn D, \label{trans3:eq}
\\
v_3&=0 &&\onn \Gamma_i, \qquad\text{for } i\in\{1,2,3,4\}.\label{trans3:BC1}
\end{align}
\end{subequations}
The problem is written this way because we want to solve it by studying the invertibility of an operator. We define the function spaces 
\[
X_s:=\{(v_1,v_2,v_3)\in H^s(D;\mathbb{R}^3)|\,v_k=0 \text{ on } \Gamma_i\quad  \text{for all}\, k,i \text{ except } k=i=1,2\},
\] 
\[
\begin{aligned}
Y_s:=\{(g_1,g_2,g_3,h_1,h_2)\in H^{s-2}(D;\mathbb{R}^3)&\times H^{s-3/2}(\Gamma_1;\mathbb{R})\times H^{s-3/2}(\Gamma_2;\mathbb{R})|\, h_1=0\text{ on }\partial\Gamma_1,\\
&h_2=0\text{ on }\partial\Gamma_2\text{ and } g_3\vert_{S_i}=0, i\in\{1,2,3,4\} \text{ if } s>3 \}
\end{aligned} 
\]
 and the operators
\[
L:(v_1,v_2,v_3)\mapsto (\Delta_{\bm{x}} v_1,\Delta_{\bm{x}} v_2,\Delta_{\bm{x}} v_3, \gamma_1\partial_{x_1}v_1,\gamma_2\partial_{x_2}v_2),
\]
\[
\begin{aligned}
&S:(v_1,v_2,v_3)\mapsto \\
&\;((\Delta_{\bm{y}}-\Delta_{\bm{x}}) v_1,(\Delta_{\bm{y}}-\Delta_{\bm{x}}) v_2,(\Delta_{\bm{y}}-\Delta_{\bm{x}}) v_3, \gamma_1(\nabla_{\bm{y}}\Phi_1\cdot\nabla_{\bm{y}}-\partial_{x_1})v_1,\gamma_2(\nabla_{\bm{y}}\Phi_2\cdot\nabla_{\bm{y}}-\partial_{x_2})v_2),
\end{aligned}
\]
where $\gamma_i$ denotes the trace operator from $D$ to $\Gamma_i$.
If we can show that
\[
L+S:X_s\to Y_s
\]
is invertible for the correct value of $s$ then we have solved problems \eqref{trans1:prob}--\eqref{trans3:prob} for a given right hand side $(g_1,g_2,g_3,h_1,h_2)\in Y_s$. However we start by showing that $L+S$ indeed maps $X_s$ into $Y_s$. That $L$ maps $X_s$ into $Y_s$ is clear but for completeness we show that $S$ does this too.
\begin{proposition}
The operator $S$ maps $X_s$ into $Y_s$.
\end{proposition}
\begin{remark} Note that it is only the conditions at $\bar{\Gamma}_1\cap\bar{\Gamma}_2$ that have to be checked since $S$ is $0$ outside $\tilde{\omega}$. In the proof we also assume $s=3+\sigma$ and check all the conditions. For smaller values of $s$ the only difference is that we only have to check the conditions that are well defined and the ones we have to check can be treated in the same way.
\end{remark}
\begin{proof}We begin by writing $\nabla_{\bm{y}}\Phi_i\cdot\nabla_{\bm{y}}$ and $\Delta_{\bm{y}}$ in terms of derivatives in the $\bm{x}$-variables
\[
\nabla_{\bm{y}}\Phi_i\cdot\nabla_{\bm{y}}=\sum_{j=1}^3a_{ij} \partial_{x_j}
\]
and 
\[
\Delta_{\bm{y}}=\sum_{i,j=1}^3 a_{ij}\partial_{x_i}\partial_{x_j}+\sum_{j=1}^3 a_j\partial_{x_j}
\]
where $a_{ij}=({\nabla_{\bm{y}}\Phi_i\cdot\nabla_{\bm{y}}\Phi_j})\circ \bm{\Phi}^{-1}$, $a_j=\Delta_{\bm{y}}\Phi_j\circ\bm{\Phi}^{-1}$.
To show that
\[
(\nabla_{\bm{y}}\Phi_1\cdot\nabla_{\bm{y}}-\partial_{x_1})v_1\vert_{S_1}=0,
\]
first note that, $\partial_{x_1}v_1\vert_{\bar{\Gamma}_1\cap\bar{\Gamma}_2}=\partial_{x_3}v_1\vert_{\bar{\Gamma}_1\cap\bar{\Gamma}_2}=0$ since $v_1=0$ on $\Gamma_2$. On the other hand, $\nabla_{\bm{y}}\Phi_i\circ\bm{\Phi}^{-1}$ is parallel to $\bm{n}_i$ on $\Gamma_i$. Hence $\nabla_{\bm{y}}\Phi_1$ and $\nabla_{\bm{y}}\Phi_2$ are orthogonal at the edge so the coefficient, $b_{12}$, in front of $\partial_{x_2}v_1$ vanishes along $S_1$.

The condition
\[
(\nabla_{\bm{y}}\Phi_2\cdot\nabla_{\bm{y}}-\partial_{x_1})v_2\vert_{S_1}=0
\]
is checked analogously.

That
\[
(\Delta_{\bm{y}}-\Delta_{\bm{x}})v_3\vert_{S_1}=0
\]
follows immediately from $v_3=0$ on both $\Gamma_1$ and $\Gamma_2$ which implies that all derivatives vanish along the edge.
\end{proof}

\subsection{Invertibility of $L+S$}
The results in Proposition \ref{Prop:DirNeu} immediately give us the following.
\begin{lemma}
Assume $s= m+\sigma$ for $m=1,2,3$. Then $L:X_s\to Y_s$ is invertible.
\end{lemma} 
This means we can instead show the invertibility of $I+L^{-1}S$. For this we will use the following lemma.
\begin{lemma}\label{lemma:invertibility} Let $X$ be a Banach space, $I:X\to X$ the identity operator and $N$ a natural number. If $T:X\to X$ is an operator so that $T^N$ is a contraction, then $I-T$ is invertible.
\end{lemma}
\begin{proof} It is well known that under the assumptions of the lemma $I-T^N$ is invertible. The invertibility of $I-T$ follows from the fact that we can write $I-T^N=(I-T)\sum_{n=0}^{N-1}T^n$.
\end{proof}
To show that the above condition holds for $L^{-1}S$ we need the following result for our change of variables.
\begin{lemma}\label{lemma:phibound}Let $\tilde{\bm{\Phi}}(y)=\bm{\Phi}(\bm{y})-\bm{y}$. Then
\begin{equation}\label{eq:phibound1}
\lim_{\epsilon\to 0}\Vert \tilde{\bm{\Phi}}\Vert_{W^{1,\infty}(T;\mathbb{R}^3)}=0.
\end{equation}
Furthermore $
\Vert \tilde{\bm{\Phi}}\Vert_{W^{2,\infty}(T;\mathbb{R}^3)}
$ remains bounded as $\epsilon\to 0$.
\end{lemma}
\begin{proof} Note that if we set $\bm{\tilde{\Psi}}(\bm{y})=\bm{\Psi}(\bm{y})-\bm{y}$ then $\tilde{\bm{\Phi}}(\bm{y})=\rho(\bm{y}/\epsilon)\bm{\tilde{\Psi}}(\bm{y})$.
For small enough $\bm{y}$ we have that
\[
\vert \bm{\tilde{\Psi}}(\bm{y})\vert \lesssim \vert \bm{y}\vert^2
\]
and
\[
\vert \partial_{y_i}\bm{\tilde{\Psi}}(\bm{y})\vert \lesssim\vert \bm{y}\vert,\; i=1,2,3.
\]
Since $\bm{\tilde{\Phi}}(\bm{y})$ has support in $[-\epsilon,\epsilon]^3\cap T$ we can choose $\epsilon$ small enough for the above inequalities to hold. Moreover
\[
\partial_{y_i}\tilde{\bm{\Phi}}(\bm{y})=\frac{1}{\epsilon}\partial_{z_i}\rho(\bm{z})\vert_{\bm{z}=\bm{y}/\epsilon}\tilde{\bm{\Psi}}(\bm{y})+ \rho(\bm{y}/\epsilon)\partial_{y_i}\bm{\tilde{\Psi}}(\bm{y}).
\]
Using the above inequalities and that the function only has support in $[-\epsilon,\epsilon]^3\cap T$ we get
\[
\Vert \tilde{\bm{\Phi}} \Vert_{L^{\infty}(T;\mathbb{R}^3)}\lesssim \epsilon^2
\]
and
\[
\Vert \partial_{y_i}\tilde{\bm{\Phi}} \Vert_{L^{\infty}(T;\mathbb{R}^3)}\lesssim \epsilon.
\]
This completes the proof of the first part. The second part is proven similarly also using that $\vert D^\alpha\tilde{\bm{\Psi}}(\bm{y})\vert$, $\vert \alpha\vert=2$, remains bounded for small $\epsilon$. 
\end{proof}
From this lemma we get the following corollary.
\begin{corollary}\label{corr:phiinvbound}Let $\bm{\Xi}(\bm{x})=\bm{\Phi}^{-1}(\bm{x})-\bm{x}$. Then
\begin{equation}\label{eq:phiinvbound}
\lim_{\epsilon\to 0}\Vert \bm{\Xi}\Vert_{W^{1,\infty}(R;\mathbb{R}^3)}=0.
\end{equation}
\end{corollary}
We also need the following result about composite functions.
\begin{lemma}\label{lemma:compositionestimate} Let $f\in C^\infty(T;\mathbb{R})$ have compact support and satisfy $f(0)=0$. Then
\[
\Vert f\circ \bm{\Phi}^{-1}\Vert_{W^{1,4}(D;\mathbb{R})}\lesssim\Vert \nabla_{\bm{y}} f\Vert_{L^\infty(T;\mathbb{R}^3)}\Vert \bm{\Phi}^{-1}\Vert_{W^{1,4}(\bm{\Phi}(\mathrm{supp}(f));\mathbb{R}^3)}
\]
\end{lemma}
\begin{proof} By a slight modification of the estimates in \cite[Remark 8]{BourdaudSickel2011} we get
\begin{align*}
\int_{\bm{\Phi}(\mathrm{supp}(f))} \vert f\circ\bm{\Phi}^{-1}(\bm{x})\vert^4 d\bm{x}&=\int_{\bm{\Phi}(\mathrm{supp}(f))} \vert f\circ\bm{\Phi}^{-1}(\bm{x})-f(0)\vert^4 d\bm{x}
\\
&\leq \Vert \nabla_{\bm{y}} f\Vert_{L^\infty(R;\mathbb{R}^3)}^4\int_{\bm{\Phi}(\mathrm{supp}(f))}\vert\bm{\Phi}^{-1}(\bm{x})\vert^4 d\bm{x}
\end{align*}
and
\begin{align*}
\int_{\bm{\Phi}(\mathrm{supp}(f))} \vert \partial_{x_i}(f\circ\bm{\Phi}^{-1})(\bm{x})\vert^4 d\bm{x}&=\int_{\bm{\Phi}(\mathrm{supp}(f))} \vert \nabla_{\bm{y}} f\circ\bm{\Phi}^{-1}(\bm{x})\cdot \partial_{x_i}\bm{\Phi}^{-1}(\bm{x})\vert^4 d\bm{x}
\\
&\leq \Vert \nabla_{\bm{y}} f\Vert_{L^\infty(R;\mathbb{R}^3)}^4\int_{\bm{\Phi}(\mathrm{supp}(f))}\vert\partial_{x_i}\bm{\Phi}^{-1}(\bm{x})\vert^4 d\bm{x},
\end{align*}
which proves the estimate.
\end{proof}
Now we can prove the following result to show that for a sufficiently small $\epsilon$ $a_{ij}-\delta_{ij}$ and $b_{ij}-\delta_{ij}$ are small in some appropriate norm.
\begin{lemma}\label{lemma:abestimate} We have that:

(i)
\[
\lim_{\epsilon\to 0}\Vert (\nabla_{\bm{y}}\Phi_i\cdot\nabla_{\bm{y}}\Phi_j-\delta_{ij})\circ \bm{\Phi}^{-1}\Vert_{W^{1,4}(D;\mathbb{R})}=0,
\]

(ii)
\[
\lim_{\epsilon\to 0} \Vert \Delta_{\bm{y}}\Phi_j\circ\bm{\Phi}^{-1}\Vert_{L^4(D;\mathbb{R})}=0.
\]
\end{lemma}
\begin{proof}(i) Since $\nabla_{\bm{y}}\Phi_i=\nabla_{\bm{y}}\tilde{\Phi}_i-\bm{e}_i$ we get
\[
\nabla_{\bm{y}}\Phi_i\cdot\nabla_{\bm{y}}\Phi_j-\delta_{ij}=\nabla_{\bm{y}}\tilde{\Phi}_i\cdot\nabla_{\bm{y}}\tilde{\Phi}_i-\partial_{\bm{y}_j}\tilde{\Phi}_i-\partial_{\bm{y}_i}\tilde{\Phi}_j,
\]
which means $[\nabla_{\bm{y}}\Phi_i\cdot\nabla_{\bm{y}}\Phi_j-\delta_{ij}](0)=0$. Moreover it has support in $[-\epsilon,\epsilon]^3\cap T$ hence we can apply Lemma \ref{lemma:compositionestimate} to get
\[
\Vert (\nabla_{\bm{y}}\Phi_i\cdot\nabla_{\bm{y}}\Phi_j-\delta_{ij})\circ \bm{\Phi}^{-1}\Vert_{W^{1,4}(D;\mathbb{R})}\leq\Vert \nabla_{\bm{y}}(\nabla_{\bm{y}}\Phi_i\cdot\nabla_{\bm{y}}\Phi_j-\delta_{ij})\Vert_{L^\infty(T;\mathbb{R}^3)}\Vert \bm{\Phi}^{-1}\Vert_{W^{1,4}(\bm{\Phi}([-\epsilon,\epsilon]^3\cap T);\mathbb{R}^3)}.
\]
The first factor, $\Vert \nabla_{\bm{y}}(\nabla_{\bm{y}}\Phi_i\cdot\nabla_{\bm{y}}\Phi_j-\delta_{ij})\Vert_{L^\infty(T;\mathbb{R}^3)}$, is bounded for small $\epsilon$ due to Lemma \ref{lemma:phibound} and the second factor can be estimated with
\begin{align*}
\Vert \bm{\Phi}^{-1}\Vert_{W^{1,4}(\bm{\Phi}([-\epsilon,\epsilon]^3\cap T);\mathbb{R}^3)}&\leq\Vert \bm{\Xi}\Vert_{W^{1,4}(\bm{\Phi}([-\epsilon,\epsilon]^3\cap T);\mathbb{R}^3)}+\Vert \mathrm{Id}\Vert_{W^{1,4}(\bm{\Phi}([-\epsilon,\epsilon]^3\cap T);\mathbb{R}^3)}
\\
&\lesssim ( \Vert \bm{\Xi}\Vert_{W^{1,\infty}(D;\mathbb{R}^3)} +\Vert \mathrm{Id}\Vert_{W^{1,\infty}(\bm{\Phi}([-\epsilon,\epsilon]^3\cap T);\mathbb{R}^3)})\Vert 1\Vert_{L^4(\bm{\Phi}([-\epsilon,\epsilon]^3\cap T);\mathbb{R})}
\end{align*}
The term $\Vert \bm{\Xi}\Vert_{W^{1,\infty}(D;\mathbb{R}^3)}$ can be made arbitrarily small by Corollary \ref{corr:phiinvbound}. For the other we note that if we pick $\delta$ as the smallest number such that $\bm{\Phi}([-\epsilon,\epsilon]^3\cap T)\subseteq [-\delta,\delta]^3\cap D$, then $\delta\to 0$ as $\epsilon \to 0$ because $\bm{\Phi}$ is continuous. This means $\Vert \mathrm{Id}\Vert_{W^{1,\infty}(\bm{\Phi}([-\epsilon,\epsilon]^3\cap T);\mathbb{R}^3)}\lesssim 1+\delta$ and $\Vert 1\Vert_{L^4(\bm{\Phi}([-\epsilon,\epsilon]^3\cap T);\mathbb{R})}\lesssim \delta^{3/4}$ hence the right hand side of the inequality above goes to $0$ as $\epsilon\to 0$, proving the first part of this lemma.

(ii) To prove the second part of this lemma we note that
\[
\begin{aligned}
\int_D\vert \Delta_{\bm{y}}\Phi_j\vert^4\circ \bm{\Phi}^{-1}(\bm{x})d\bm{x}&=\int_D\vert\Delta_{\bm{y}}\Phi_j(\bm{y})\vert^4\vert\det D\bm{\Phi(y)}\vert d\bm{y}\\
&\leq\Vert \Delta_{\bm{y}}\Phi_j\Vert_{L^\infty(T;\mathbb{R}^3)}\int_{[-\epsilon,\epsilon]^3\cap T}\vert\det D\bm{\Phi(y)}\vert d\bm{y}
\end{aligned}
\]
because $\supp \Delta_{\bm{y}}\Phi_j\subset [-\epsilon,\epsilon]^3\cap T$. Furthermore $\Vert \Delta_{\bm{y}}\Phi_j\Vert_{L^\infty(T;\mathbb{R}^3)}$ is bounded as $\epsilon\to 0$ by Lemma \ref{lemma:phibound} since $\Delta_{\bm{y}}\Phi_j=\Delta_{\bm{y}}\tilde{\Phi}_j$. The same lemma shows that $D\bm{\Phi}\to \mathrm{Id}$ as $\epsilon\to 0$. Hence $\det D\bm{\Phi}\to 1$ as $\epsilon\to 0$ which means
\[
\int_{[-\epsilon,\epsilon]^3\cap T}\vert \det D\bm{\Phi}(\bm{y})\vert d\bm{y}\lesssim (1+\Vert \det D\bm{\Phi}(\bm{y})-1\Vert_{L^\infty([-\epsilon,\epsilon]^3\cap T,\mathbb{R})}) \epsilon^3,
\]
where the term $\Vert \det D\bm{\Phi}(\bm{y})-1\Vert_{L^\infty([-\epsilon,\epsilon]^3\cap T,\mathbb{R})}\to 0$ as $\epsilon\to 0$. This proves the second part of the lemma.
\end{proof}
Finally we can prove the following proposition, which together with Lemma \ref{lemma:invertibility} shows that $L+S$ is invertible.
\begin{proposition}\label{prop:contraction} Let $s=m+\sigma$.

(i) If $m=1$ and $\epsilon$ is sufficiently small then $L^{-1}S:X_s\to X_s$ is a contraction.

(ii) If $m=2,3$ and $\epsilon$ is sufficiently small then there exists $N$ such that $(L^{-1}S)^N:X_s\to X_s$ is a contraction.
\end{proposition}
\begin{proof}(i) We begin by estimating the norm of $S$:
\begin{align*}
\Vert S \bm{u}\Vert_{Y_s}&\lesssim\sum_{k=1}^3\sum_{i,j=1}^3 \Vert  (a_{ij}-\delta_{ij})\partial_{x_i}\partial_{x_j} u_k\Vert_{H^{s-2}(D;\mathbb{R})}+\sum_{k=1}^3\sum_{j=1}^3\Vert a_j\partial_{x_j}u_k\Vert_{H^{s-2}(D;\mathbb{R})}
\\
&\qquad+\sum_{k=1}^2 \sum_{i=1}^3 \Vert (a_{ik}-\delta_{ik})\partial_{x_i} u_k \Vert_{H^{s-3/2}(\Gamma_k;\mathbb{R})}
\\
&\lesssim\sum_{k=1}^3\sum_{i,j=1}^3 \Vert  (a_{ij}-\delta_{ij})\partial_{x_i}\partial_{x_j} u_k\Vert_{H^{s-2}(D;\mathbb{R})}+\sum_{k=1}^3\sum_{j=1}^3\Vert a_j\partial_{x_j}u_k\Vert_{H^{s-2}(D;\mathbb{R})}
\\
&\qquad+\sum_{k=1}^2 \sum_{i=1}^3 \Vert (a_{ik}-\delta_{ik})\partial_{x_i} u_k \Vert_{H^{s-1}(D;\mathbb{R})}.
\end{align*}
By Theorems $3$, $5$ and $6$ in \cite{Sickel1993} we have that
\[
\begin{aligned}
\Vert  (a_{ij}-\delta_{ij})\partial_{x_i}\partial_{x_j} u_k\Vert_{H^{s-2}(D;\mathbb{R})}&\lesssim \Vert  (a_{ij}-\delta_{ij}) \Vert_{W^{1,4}(D;\mathbb{R})}\Vert\partial_{x_i}\partial_{x_j} u_k\Vert_{H^{s-2}(D;\mathbb{R})}
\\
&\lesssim \Vert  (a_{ij}-\delta_{ij}) \Vert_{W^{1,4}(D;\mathbb{R})}\Vert\bm{u}\Vert_{X_s},
\end{aligned}
\]
\[
\begin{aligned}
\Vert a_j\partial_{x_j}u_k\Vert_{H^{s-2}(D;\mathbb{R})}&\lesssim \Vert a_j\Vert_{L^4(D;\mathbb{R})}\Vert \partial_{x_j}u_k\Vert_{H^{s-1}(D;\mathbb{R})}
\\
&\lesssim \Vert a_j\Vert_{L^4(D;\mathbb{R})}\Vert \bm{u}\Vert_{X_s},
\end{aligned}
\]
and
\[
\begin{aligned}
\Vert  (a_{ik}-\delta_{ik})\partial_{x_i} u_k\Vert_{H^{s-1}(D;\mathbb{R})}&\lesssim \Vert  (a_{ik}-\delta_{ik}) \Vert_{W^{1,4}(D;\mathbb{R})}\Vert\partial_{x_i} u_k\Vert_{H^{s-1}(D;\mathbb{R})}
\\
&\lesssim \Vert  (a_{ik}-\delta_{ik}) \Vert_{W^{1,4}(D;\mathbb{R})}\Vert\bm{u}\Vert_{X_s}.
\end{aligned}
\]
By Lemma \ref{lemma:abestimate} $\Vert  (a_{ij}-\delta_{ij}) \Vert_{W^{1,4}(D;\mathbb{R})}$ and $\Vert a_j\Vert_{L^4(D;\mathbb{R})}$ can be made arbitrarily small for small enough $\epsilon$. Hence the norm of $S$ can be made arbitrarily small. This means that the norm of $L^{-1}S$ can be made small enough that $L^{-1}S$ is a contraction.

(ii) For $m=2$ the terms of the form $\Vert  (a_{ij}-\delta_{ij})\partial_{x_i}\partial_{x_j} u_k\Vert_{H^{s-2}(D;\mathbb{R})}$ can be estimated as before, while we have to estimate the boundary terms in a different way. First we note that
\begin{align*}
\Vert  (a_{ik}-\delta_{ik})\partial_{x_i} u_k\Vert_{H^{s-1}(D;\mathbb{R})}&\lesssim \sum_{j=1}^3 \left[\Vert  \partial_{x_j}(a_{ik}-\delta_{ik})\partial_{x_i} u_k\Vert_{H^{s-2}(D;\mathbb{R})}\right.
\\
&\qquad+\left.\Vert  (a_{ik}-\delta_{ik})\partial_{x_j}\partial_{x_i} u_k\Vert_{H^{s-2}(D;\mathbb{R})}\right]
\\
&\qquad+\Vert  (a_{ik}-\delta_{ik})\partial_{x_i} u_k\Vert_{H^{s-2}(D;\mathbb{R})}
\end{align*}
The term $\Vert  (a_{ik}-\delta_{ik})\partial_{x_k}\partial_{x_i} u_k\Vert_{H^{s-2}(D;\mathbb{R})}$ can be estimated above with $\delta_1\Vert \bm{u}\Vert_{X_s}$ where $\delta_1$ can be made arbitrarily small and the other terms can be estimated with $C_\epsilon\Vert \bm{u}\Vert_{X_{s-1}}$ where $C_\epsilon$ is a constant of depending on $\epsilon$. We can also estimate the terms of the form $\Vert a_j\partial_{x_j}u_k\Vert_{H^{s-2}(D;\mathbb{R})}$ with $C_\epsilon\Vert \bm{u}\Vert_{X_{s-1}}$. This means
\[
\Vert L^{-1}S\bm{u}\Vert_{X_s}\leq\Vert L^{-1}\Vert_{\mathcal{L}(Y_s,X_s)}(\delta_1\Vert \bm{u}\Vert_{X_s}+C_\epsilon\Vert \bm{u}\Vert_{X_{s-1}})
\]
If $\epsilon$ is chosen small enough that the norm of $L^{-1}S:X_{s-1}\to X_{s-1}$ is less than $\Vert L^{-1}\Vert_{\mathcal{L}(Y_s,X_s)}\delta_1$, which we know we can do from part $(i)$, then it follows that
\[
\Vert (L^{-1}S)^n\bm{u}\Vert_{X_s}\leq\Vert L^{-1}\Vert_{\mathcal{L}(Y_s,X_s)}^n(\delta_1^n\Vert \bm{u}\Vert_{X_s}+C_\epsilon n\delta_1^{n-1}\Vert \bm{u}\Vert_{X_{s-1}}).
\]
We prove this by induction. For $n=1$ it is clearly true. Assume it holds for $n-1$ then
\[
\Vert (L^{-1}S)^n\bm{u}\Vert_{X_s}\leq\Vert L^{-1}\Vert_{\mathcal{L}(Y_s,X_s)}(\delta_1\Vert (L^{-1}S)^{n-1}\bm{u}\Vert_{X_s}+C_\epsilon\Vert (L^{-1}S)^{n-1}\bm{u}\Vert_{X_{s-1}})
\]
where
\[
\Vert L^{-1}\Vert_{\mathcal{L}(Y_s,X_s)}\delta_1\Vert (L^{-1}S)^{n-1}\bm{u}\Vert_{X_s}\leq \Vert L^{-1}\Vert_{\mathcal{L}(Y_s,X_s)}^n(\delta_1^n\Vert \bm{u}\Vert_{X_s}+C_\epsilon (n-1)\delta_1^{n-1}\Vert \bm{u}\Vert_{X_{s-1}})
\]
by the induction assumption and 
\[
\Vert L^{-1}\Vert_{\mathcal{L}(Y_s,X_s)}C_\epsilon\Vert (L^{-1}S)^{n-1}\bm{u}\Vert_{X_{s-1}}\leq \Vert L^{-1}\Vert_{\mathcal{L}(Y_s,X_s)}^n C_\epsilon\delta_1^{n-1}\Vert \bm{u}\Vert_{X_{s-1}}
\]
by part $(i)$ of the proposition.
Since $\Vert \bm{u}\Vert_{X_{s-1}}\leq\Vert \bm{u}\Vert_{X_s}$ we get
\[
\Vert (L^{-1}S)^n\bm{u}\Vert_{X_s}\leq\Vert L^{-1}\Vert_{\mathcal{L}(Y_s,X_s)}^n(\delta_1^n+C_\epsilon n\delta_1^{n-1})\Vert \bm{u}\Vert_{X_s}
\]
If $\delta_1<\frac{1}{\Vert L^{-1}\Vert_{\mathcal{L}(Y_s,X_s)}}$ then clearly $\Vert L^{-1}\Vert_{\mathcal{L}(Y_s,X_s)}^n(\delta_1^n+C_\epsilon n\delta_1^{n-1})\to 0$ as $n\to \infty$ and there must exist some $N$ such that $(L^{-1}S)^N:X_s\to X_s$ is a contraction. 

For $m=3$ we have to use the same trick for the $\Vert (a_{ij}-\delta_{ij})\partial_{x_i}\partial_{x_j}u_k\Vert_{H^{s-2}(D;\mathbb{R})}$ terms. That is,
\[
\begin{aligned}
\Vert (a_{ij}-\delta_{ij})\partial_{x_i}\partial_{x_j}u_k\Vert_{H^{s-2}(D;\mathbb{R})}&\lesssim\sum_{l=1}^3\left[\Vert (a_{ij}-\delta_{ij})\partial_{x_l}\partial_{x_i}\partial_{x_j}u_k\Vert_{H^{s-3}(D;\mathbb{R})}\right.
\\
&\qquad\left.+\Vert \partial_{x_l}(a_{ij}-\delta_{ij})\partial_{x_i}\partial_{x_j}u_k\Vert_{H^{s-3}(D;\mathbb{R})}\right]
\\
&\qquad+\Vert (a_{ij}-\delta_{ij})\partial_{x_i}\partial_{x_j}u_k\Vert_{H^{s-3}(D;\mathbb{R})}.
\end{aligned}
\]
We estimate the terms of the first type in the following way
\[
\Vert (a_{ij}-\delta_{ij})\partial_{x_l}\partial_{x_i}\partial_{x_j}u_k\Vert_{H^{s-3}(D;\mathbb{R})}\lesssim \Vert a_{ij}-\delta_{ij}\Vert_{W^{1,4}(D;\mathbb{R})}\Vert \bm{u}\Vert_{X_s},
\]
where $\Vert a_{ij}-\delta_{ij}\Vert_{W^{1,4}(D;\mathbb{R})}$ can be made arbitrarily small. All the remaining terms can be estimated with $C_\epsilon\Vert \bm{u}\Vert_{X_{s-1}}$. Similarly, we can estimate the boundary terms by terms of the form $\Vert (a_{ij}-\delta_{ij})\partial_{x_l}\partial_{x_i}\partial_{x_j}u_k\Vert_{H^{s-3}(D;\mathbb{R})}$ and terms that can be estimated by $C_\epsilon\Vert \bm{u}\Vert_{X_{s-1}}$. The terms of the form $\Vert a_j\partial_{x_j}u_k\Vert_{H^{s-2}(D;\mathbb{R})}$ can also be estimated with $C_\epsilon\Vert \bm{u}\Vert_{X_{s-1}}$. Then we can use that
\[
\Vert \bm{u}\Vert_{X_{s-1}}\leq \xi \Vert \bm{u}\Vert_{X_s}+C_\xi \Vert \bm{u}\Vert_{X_{s-2}}
\]
holds for any $\xi>0$ (see e.g. Theorem 1.4.3.3 in \cite{Grisvard2011}). This allows us to prove 
\[
\Vert (L^{-1}S)^n\bm{u}\Vert_{X_s}\leq\Vert L^{-1}\Vert_{\mathcal{L}(Y_s,X_s)}^n(\delta_1^n\Vert \bm{u}\Vert_{X_s}+C_\epsilon n\delta_1^{n-1}\Vert \bm{u}\Vert_{X_{s-2}})
\]
for any $\delta_1>0$ (depending on $\epsilon$) using essentially the same induction argument as in the case $m=2$. This inequality proves the claim for $m=3$.
\end{proof}

\subsection{Weak Solution}
The invertibility of $L+S$ allows us to find a solution to \eqref{trans1:prob}--\eqref{trans3:prob} given an element in $Y_s$. However the element in $Y_s$ for which we want to find a solution depends on the solution to \eqref{orig:prob} itself. It does so through lower order terms though and thus the invertibility of $L+S$ can be used to find a solution of higher regularity. This requires a starting point and thus we need a weak solution.
\begin{definition}\label{def:weaksol}
An element 
\[
(\tilde{u}_1,\tilde{u}_2,\tilde{u}_3)\in Z_1,
\] where 
\[Z_s:=\{(v_1,v_2,v_3)\in H^s(T;\mathbb{R}^3)|v_k=0 \text{ on } R_i\quad  \forall\, k,i \text{ except } k=i=1,2\},\] 
is said to be a weak solution of the problem \eqref{loc1:prob}--\eqref{loc3:prob} if it satisfies
\begin{equation}\label{weaksol}
\sum_{k=1}^3\int_T\nabla_{\bm{y}} \tilde{u}_k\nabla_{\bm{y}} w_k dy=\sum_{i=1}^2\int_{R_i} h_iw_i dS(y)-\sum_{k=1}^3\int_T g_k w_k dy
\end{equation}
for all $(w_1,w_2,w_3)\in Z_1$.
\end{definition}
That such a solution exists is easily proven using the Riesz representation theorem.
\begin{lemma}\label{lemma:transweaksol}
If $g_k\in L^2(T;\mathbb{R})$, $k=1,2,3$ and $h_i\in L^2(R_i;\mathbb{R})$, $i=1,2$ then the problems \eqref{loc1:prob}--\eqref{loc3:prob} have a unique weak solution $(\tilde{u}_1,\tilde{u}_2,\tilde{u}_3)\in X_1$.
\end{lemma}

It is also known \cite[Theorem 10.1]{Zajaczkowski1988} that the original problem \eqref{orig:prob} has a weak solution.
\begin{lemma}\label{lemma:weaksol} There exists a unique function $\bm{u}$ in $H^1(\Omega;\mathbb{R}^3)$ such that 
$\bm{u}\cdot\bm{\tau}^{(i)}_j=0$ on $\partial\Omega_i$ for all i and j=1,2, and
\begin{equation}\label{zweaksol}
\sum_{i=1}^3\int_\Omega \nabla_{\bm{y}} \bm{u}\cdot \bm{e}_i\cdot\nabla_{\bm{y}} \bm{w}\cdot \bm{e}_i dy=-\int_{\partial\Omega} (\bm{n}\cdot \bm{u})(\bm{n}\cdot \bm{w}) \nabla_{\bm{y}}\cdot{\bm{n}}dS(y)-\int_\Omega \bm{f}\cdot\bm{w}dy
\end{equation}
for all $\bm{w}\in H^1(\Omega;\mathbb{R}^3)$ such that
$\bm{w}\cdot\bm{\tau}^{(i)}_j=0$ on $\partial\Omega_i$ for i=1,2,3 and j=1,2, and $\bm{f}$ satisfying $\divv \bm{f}=0$.
Moreover
\begin{equation}\label{eq:weaksolest}
\Vert \bm{u}\Vert_{H^1(\Omega;\mathbb{R}^3)}\lesssim \Vert \bm{f}\Vert_{L^2(\Omega;\mathbb{R}^3)}.
\end{equation}
\end{lemma}
Inserting this solution as $\bm{u}$ (extended arbitrarily since every term involving $\bm{u}$ is multiplied with some smooth function with support in $\omega$) on the right hand sides of problems \eqref{loc1:prob}--\eqref{loc3:prob} gives us $g_k\in L^2(T;\mathbb{R}), k=1,2,3$ and $h_i\in L^2(R_i;\mathbb{R})$, $i=1,2$ so by Lemma \ref{lemma:transweaksol} the problems have a weak solution. It is natural to suspect this solution to be $(\eta\bm{u}\cdot \bm{n}_1,\eta\bm{u}\cdot \bm{n}_2,\eta\bm{u}\cdot \bm{n}_3)$ and that this is indeed the case can be shown by direct calculation.
\begin{lemma}\label{lemma:equivweaksol}
The weak solution of \eqref{loc1:prob}--\eqref{loc3:prob} obtained by inserting the the weak solution of \eqref{orig:prob} in the terms on the right hand sides is $(\eta\bm{u}\cdot \bm{n}_1,\eta\bm{u}\cdot \bm{n}_2,\eta\bm{u}\cdot \bm{n}_3)$.
\end{lemma}

\subsection{Higher Regularity}
The only step remaining to show that our weak solution has the desired regularity.
\begin{proposition}\label{Prop:Yspace} Let $\bm{f}\in H^{s-1}(\Omega;\mathbb{R}^3)$ be a function satisfying $\bm{n}_3\cdot\bm{f}=0$ in $\omega$ along the edge $\bar{\partial\Omega_1}\cap\bar{\partial\Omega_2}$. If we have a solution  $\bm{u}\in H^s(\Omega;\mathbb{R}^3)$ of \eqref{orig:prob} and insert it into \eqref{eq:gkx} and \eqref{eq:hkx} we get $(g_1^{(x)},g_2^{(x)},g_3^{(x)},h_1^{(x)},h_2^{(x)})\in Y_{s+1}$.
\end{proposition}
\begin{proof}
That we have the correct regularity is trivial, so we only have to check the conditions. As before we assume enough regularity for all conditions to be well defined. If the regularity is lower the only difference is that we have to check less conditions.
The support of $\eta$ also allows us to only be concerned with the conditions along $S_1$. 

The fact that $\bm{u}= 0$ along the edge immediately implies that $h_1^{(x)}=h_2^{(x)}=0$ there since they only depend on $\bm{u}$ and not its derivatives.

Showing that $g_3^{(x)}=0$ along the edge is not quite as trivial. Recall that $g_3^{(x)}=g_3\circ\bm{\Phi}^{-1}$ where
\[
g_3=(\bm{\tilde{f}}+2(\nabla_{\bm{y}}\eta\cdot\nabla_{\bm{y}})\bm{u}+ \bm{u}\Delta_{\bm{y}}\eta)\cdot \bm{n}_3+2\sum_i\partial_{y_i}\bm{\tilde{u}}\cdot\partial_{y_i}\bm{n}_3+\bm{\tilde{u}}\cdot\Delta_{\bm{y}} \bm{n}_3.
\]
We let the composition with $\bm{\Phi}^{-1}$ be implicit in the following and show that each term vanishes along the edge.
We have $\tilde{\bm{f}}\cdot\bm{n}_3=0$ by assumption.
Since $\bm{u}=0$ along the edge and all the derivatives of  $\bm{u}\cdot\bm{n}_3$ vanish along the edge we get
\[
(\nabla_{\bm{y}}\eta\cdot\nabla_{\bm{y}})(\bm{u})\cdot\bm{n}_3=(\nabla_{\bm{y}}\eta\cdot\nabla_{\bm{y}})(\bm{u}\cdot\bm{n}_3)=0.
\]
Moreover, $\bm{u}\cdot \bm{n}_3\Delta_{\bm{y}}\eta=0$ and $\bm{\tilde{u}}\cdot\Delta_{\bm{y}} \bm{n}_3=0$ since $\bm{u}=0$ along the edge. 
We are left with the term
\[
\sum_i\partial_{y_i}\bm{\tilde{u}}\cdot\partial_{y_i}\bm{n}_3=\sum_{i,j}a_{ij}\partial_{x_i}\bm{\tilde{u}}\cdot\partial_{x_j}\bm{n}_3.
\]
Using that $\bm{u}=0$ on the edge we get
\[
\sum_{i,j}a_{ij}\partial_{x_i}\bm{\tilde{u}}\cdot\partial_{x_j}\bm{n}_3=\sum_{i,j,k}a_{ij}\partial_{x_i}(v_k\bm{n}_k)\cdot\partial_{x_j}\bm{n}_3=\sum_{i,j,k}\bm{n}_k\cdot a_{ij}\partial_{x_j}\bm{n}_3\partial_{x_i}v_k.
\]
However, all the derivatives of the $v_k$:s vanish at the edge except $\partial_{x_1}v_2$ and $\partial_{x_2}v_1$. So the only remaining terms are
\[
\sum_{j}\bm{n}_1\cdot a_{2j}\partial_{x_j}\bm{n}_3\partial_{x_2}v_1+\sum_{j}\bm{n}_2\cdot a_{1j}\partial_{x_j}\bm{n}_3\partial_{x_1}v_2.
\]
We show how the first term vanishes as both can be treated similarly. Note that
\[
\sum_j a_{2j}\partial_{x_j}=\nabla_{\bm{y}}\Phi_2\cdot\nabla_{\bm{y}},
\]
where $\nabla_y \bm{\Phi}_2$ is parallel to $\bm{n}_2$ at the edge. This means the factor in front of $\partial_{x_2}v_1$ is proportional to $\bm{n_1}(\bm{n}_2\cdot\nabla_{\bm{y}})\bm{n}_3$. Furthermore, since $\bm{n}_1\cdot \bm{n}_3=0$ we get
\[
\bm{n_1}\cdot[(\bm{n}_2\cdot\nabla_{\bm{y}})\bm{n}_3]=-\bm{n_3}\cdot[(\bm{n}_2\cdot\nabla_{\bm{y}})\bm{n}_1].
\]
The right hand side here is precisely
\[
-\sigma_1(\bm{n}_2)\cdot\bm{n}_3=-\sigma_1(\bm{n}_3)\cdot\bm{n}_2
\]
where $\sigma_1(\bm{z})=\bm{z}\cdot\nabla_{\bm{y}}\bm{n_1}$ is the shape operator for $R_1$. This vanishes since a curve with unit tangent $\bm{\tau}$ satisfies $S_1(\bm{\tau})=\lambda\bm{\tau}$, for some $\lambda$, along the curve if and only if it is a curvature line. Recall that by the the definition of $\Omega$ in the introduction $\overline{\partial\Omega}_1\cap\overline{\partial\Omega}_2$ is a curvature line and $\bm{n}_3$ is the unit tangent. Hence
\[
-\sigma_1(\bm{n}_3)\cdot\bm{n}_2=-\lambda \bm{n}_3\cdot\bm{n}_2=0,
\]
which concludes the proof.
\end{proof}
Finally we can prove the following theorem, which gives us a solution $\bm{u}$ to $\eqref{orig:prob}$. With this solution we then obtain a solution $\bm{w}$ to the $(\mathcal{DCP})$ by setting $\bm{w}=-\curl\bm{u}$.
\begin{theorem}\label{thm:divcurl}If $\bm{f}\in H^{\sigma+1}(\Omega;\mathbb{R}^3)$ such that $\bm{f}\cdot \bm{\hat{\tau}}_\pm\vert_{\overline{\partial\Omega}_\pm\cap\overline{\partial\Omega}_0}=0$ where $\bm{\hat{\tau}}_\pm$ is the tangent vector to $\overline{\partial\Omega}_\pm\cap\overline{\partial\Omega}_0$, then there exists a unique solution $\bm{u}\in H^{\sigma+3}(\Omega;\mathbb{R}^3)$ to \eqref{orig:prob}.
\end{theorem}
\begin{proof}Begin by covering $\Omega$ with open sets $\{\omega_{\bm{y}}\}_{\bm{y}\in \Omega}$ small enough for $\omega_{\bm{y}}$, $U_{\bm{y}}$, and $\epsilon_{\bm{y}}$ to satisfy all the properties required for the results above to be applicable. Then extract a finite subcover $\{\omega_{\bm{y}_i}\}_{i= 1}^m$ for which we define an associated partition of unity $\{\eta_{\bm{y}_i}\}_{i=1}^m$.

By Lemma \ref{lemma:weaksol} we have a unique weak solution $\bm{u}\in H^1(\Omega)$. For every $ 1\leq i\leq m$ for which $\omega_{y_i}$ intersects the edges we insert this solutions in the expressions for $h_k^{(x)}$ and $g_k^{(x)}$ in equations \eqref{eq:gkx} and \eqref{eq:hkx}. This gives an element in $Y_{\sigma+1}$ by Proposition \ref{Prop:Yspace}. Hence the equations are solvable with a solution in $X_{\sigma+1}$ by Lemma \ref{lemma:invertibility} and Proposition \ref{prop:contraction}. 

This solution is a weak solution to problems \eqref{trans1:prob}--\eqref{trans3:prob} even though it does not have sufficient regularity to directly satisfy the equations. However it does satisfy the equations in a distributional sense which means that the change of variables can be treated as usual and we get a weak solution to \eqref{loc1:prob}--\eqref{loc3:prob}, $\tilde{\bm{u}}^{(i)}\in H^{\sigma+1}(T;\mathbb{R}^3)$. This regularity combined with $\Delta_{\bm{y}} \tilde{\bm{u}}_i\in L^2(T;\mathbb{R}^3)$ is sufficient to apply Green's identity so it is trivial to check that it is indeed the weak solution given in Definition \ref{def:weaksol}. The sum of these solutions and similar ones for the $i$ where $\omega_i$ does not intersect the edge (which can be obtained by standard theory for partial differential equations) equals the weak solution by Lemma \ref{lemma:equivweaksol}, i.e we have a unique solution in $H^{\sigma+1}(\Omega)$. 

The argument given above can then be repeated for $\sigma+2$ and $\sigma+3$ to prove the theorem.
\end{proof}
\subsection{Estimates}
Here we give a proof for the inequalities $\eqref{dcpest1}$ and $\eqref{dcpest2}$. We begin with $\eqref{dcpest1}$. It is clear that 
\[
\Vert \bm{w}\Vert_{H^{\sigma+2}(\Omega;\mathbb{R}^3)}\lesssim \Vert \bm{u}\Vert_{H^{\sigma+3}(\Omega;\mathbb{R}^3)}.
\]
Furthermore, keeping the notation from Theorem \ref{thm:divcurl} and its proof, we have
\[
\Vert \bm{u}\Vert_{H^{\sigma+3}(\Omega;\mathbb{R}^3)}\lesssim \sum_{i=1}^m\Vert \tilde{\bm{u}}^{(i)}\Vert_{H^{\sigma+3}(\omega_i;\mathbb{R}^3)}.
\]
For all the $\omega_i$ which do not intersect the edges of $\Omega$ standard elliptic regularity gives us $\Vert \tilde{\bm{u}}^{(i)}\Vert_{H^{\sigma+3}(\omega_i;\mathbb{R}^3)}\lesssim \Vert \bm{f}\Vert_{H^{\sigma+1}(\Omega;\mathbb{R}^3)}+\Vert \bm{u}\Vert_{H^{\sigma+2}(\omega_i;\mathbb{R}^3)}$.
We want a similar estimate for the $\omega_i$ which do intersect the edges. For this we drop the subscript $i$ and use the notation introduced above. It is clear that 
\[
\Vert \tilde{\bm{u}}\Vert_{H^{\sigma+3}(\omega;\mathbb{R}^3)}\lesssim \Vert \bm{v}\Vert_{H^{\sigma+3}(D;\mathbb{R}^3)}=\Vert \bm{v}\Vert_{X_{\sigma+3}}.
\]
Since $L+S:X_{3+\sigma}\to Y_{3+\sigma}$ is bounded and bijective it has a bounded inverse hence
\[
\Vert \bm{v}\Vert_{X_{\sigma+3}}\lesssim \Vert(g_1^{(x)},g_2^{(x)},g_3^{(x)},h_1^{(x)},h_2^{(x)})\Vert_{Y_{3+\sigma}}.
\]
By definition we can estimate
\[
\Vert(g_1^{(x)},g_2^{(x)},g_3^{(x)},h_1^{(x)},h_2^{(x)})\Vert_{Y_{3+\sigma}}\lesssim \Vert\bm{u}\Vert_{H^{\sigma+2}(\Omega;\mathbb{R}^3)}+\Vert \bm{f}\Vert_{H^{\sigma+1}(\Omega;\mathbb{R}^3)}.
\]
Hence
\[
\Vert \bm{u}\Vert_{H^{\sigma+3}(\Omega;\mathbb{R}^3)}\lesssim\Vert\bm{u}\Vert_{H^{\sigma+2}(\Omega;\mathbb{R}^3)}+\Vert \bm{f}\Vert_{H^{\sigma+1}(\Omega;\mathbb{R}^3)}.
\]
Combining this with Theorem 1.4.3.3 in \cite{Grisvard2011} and \eqref{eq:weaksolest} we get
\[
\Vert \bm{u}\Vert_{H^{\sigma+3}(\Omega;\mathbb{R}^3)}\lesssim\Vert \bm{f}\Vert_{H^{\sigma+1}(\Omega;\mathbb{R}^3)},
\]
proving \eqref{dcpest1}.

The proof of \eqref{dcpest2} is very similar if we define $X_2$ and $Y_2$ through interpolation. However the interpolation space between $H^{\sigma-1/2}(\Gamma_i;\mathbb{R})$ and $H^{\sigma+1/2}(\Gamma_i;\mathbb{R})$ is slightly more complicated than $H^{1/2}(\Gamma_i;\mathbb{R})$, but we can always estimate its norm by the norm of $H^{\sigma+1/2}(\Gamma_i;\mathbb{R})$. This gives the estimate
\begin{align*}
\Vert(g_1^{(x)},g_2^{(x)},g_3^{(x)},h_1^{(x)},h_2^{(x)})\Vert_{Y_{2}}&\lesssim \Vert(g_1^{(x)},g_2^{(x)},g_3^{(x)})\Vert_{L^{2}(\Omega;\mathbb{R}^3)}+\Vert h_1^{(x)}\Vert_{H^{\sigma+1/2}(\Gamma_1;\mathbb{R})}
\\
&\qquad+\Vert h_2^{(x)}\Vert_{H^{\sigma+1/2}(\Gamma_2;\mathbb{R})}\\
&\lesssim\Vert\bm{u}\Vert_{H^{\sigma+1}(\Omega;\mathbb{R}^3)}+\Vert \bm{f}\Vert_{L^2(\Omega;\mathbb{R}^3)}.
\end{align*}
Thus
\[
\Vert \bm{u}\Vert_{H^{2}(\Omega;\mathbb{R}^3)}\lesssim\Vert\bm{u}\Vert_{H^{\sigma+1}(\Omega;\mathbb{R}^3)}+\Vert \bm{f}\Vert_{L^2(\Omega;\mathbb{R}^3)},
\]
which implies
\[
\Vert \bm{u}\Vert_{H^{2}(\Omega;\mathbb{R}^3)}\lesssim\Vert \bm{f}\Vert_{L^2(\Omega;\mathbb{R}^3)}
\]
in the same way as before and the proof is completed.
\section{Transport Problem}\label{Sec:TP}
This entire section is the proof of Theorem \ref{Theorem:TP}. 

We begin with the claim that the solution satisfies $\divv \bm{f}=0$.
The argument is the same as the one presented in the proof of Lemma $2.2$ in \cite{Alber1992}. Assume we have a solution as stated in the theorem. By the same calculations as in \cite{Alber1992} we get that $\divv \bm{f}$ satisfies
\begin{align}
(\bm{v}\cdot\nabla) \divv\bm{f}&=0 \qquad\inn \Omega,\label{divfeq}\\
\divv\bm{f}&=0 \qquad\onn \partial\Omega_-.
\end{align}
By Lemma \ref{Lemma:IntegralCurves} we know that the integral curves cover $\Omega$ and that they all intersect $\partial\Omega_-$. Hence we can change variables to a parametrisation of the inflow set and one variable going along the integral curves, which we denote by $t$. In these variables the equations become
\begin{align}
\partial_t \divv \bm{f}&=0\qquad\text{ for } t>0,\label{divfeq2}\\
\divv\bm{f}&=0 \qquad\text{ for } t=0.
\end{align}
We also note that the regularity of $\bm{f}$ is not high enough for derivatives of the $\divv\bm{f}$ to be well defined as functions in general. However, it is possible to show that \eqref{divfeq2} holds in a distributional sense. That the solution is unique follows from Theorem $3.19$ in \cite{BahouriCheminDanchin11}. Hence $\divv \bm{f}=0$
\subsection{An Auxiliary Result}
We begin by solving the auxiliary problem 
\begin{equation}\label{transport}
\begin{aligned}
\partial_t \bm{f}+(\bm{v}\cdot \nabla) \bm{f} + A\bm{f}&=\bm{g}\qquad && \inn [0,T]\times \mathbb{R}^2,\\
\bm{f}&=\bm{f}_0\qquad && \onn \mathbb{R}^2,
\end{aligned}
\end{equation}
that is, finding the function $\bm{f}:[0,T]\times\mathbb{R}^2\to \mathbb{R}^3$ if the functions $\bm{v}:[0,T]\times\mathbb{R}^2\to \mathbb{R}^3$, $A:[0,T]\times\mathbb{R}^2\to \mathbb{R}^{3\times3}$, $\bm{g}:[0,T]\times\mathbb{R}^2\to\mathbb{R}^3$, and $\bm{f}_0:\mathbb{R}^2\to\mathbb{R}^3$ are given. The argument for $f$ is separated into the `time' variable, $t$, and the `space' variables $\bm{x}=(x_1,x_2)$, i.e. $\bm{f}:(t,\bm{x})=(t,x_1,x_2)\mapsto \bm{f}=(f_1,f_2,f_3)$.
\begin{proposition}\label{prop:transport1} If $\bm{v}\in H^{\sigma+2}([0,T]\times\mathbb{R}^2;\mathbb{R}^3)$, $A\in H^{\sigma+1}([0,T]\times\mathbb{R}^2;\mathbb{R}^{3\times 3})$, $\bm{g}\in H^{\sigma+1}([0,T]\times\mathbb{R}^2;\mathbb{R}^3)$, and $\bm{f}_0\in H^{\sigma+1}(\mathbb{R}^2;\mathbb{R}^3)$, then there exists a unique solution \\$\bm{f}\in C([0,T];H^{\sigma+1}(\mathbb{R}^2;\mathbb{R}^3))\cap H^{\sigma+1}([0,T]\times\mathbb{R}^2;\mathbb{R}^3)$ to \eqref{transport}.
\end{proposition}
\begin{proof} To find a unique solution $\bm{f}\in C([0,T];H^s(\mathbb{R}^2;\mathbb{R}^3))$ we apply Theorem 3.19 in \cite{BahouriCheminDanchin11} (see also Definition $3.13$ in \cite{BahouriCheminDanchin11} for the definition  of a solution which is not necessarily differentiable with respect to $t$). To apply this theorem have to show that 
\begin{equation}\label{thmproof1}
(f_0)_i\in B^{\sigma+1}_{2,2}(\mathbb{R}^2;\mathbb{R})
\end{equation}
\begin{equation}\label{thmproof2}
g_i\in L^1([0,T];B^{\sigma+1}_{2,2}(\mathbb{R}^2;\mathbb{R})),
\end{equation}
\begin{equation}\label{thmproof3}
v_i\in L^\rho([0,T];B^{-M}_{\infty ,\infty}(\mathbb{R}^2;\mathbb{R})),
\end{equation}
for $i=1,2,3$, some $\rho>1$ and $M>0$, and 
\begin{equation}\label{thmproof4}
\nabla \bm{v}\in L^1([0,T];B^1_{2,\infty}(\mathbb{R}^2;\mathbb{R})\cap L^\infty(\mathbb{R}^2;\mathbb{R})).
\end{equation}
Equation \eqref{thmproof1} follows from $H^{\sigma+1}(\mathbb{R}^2;\mathbb{R})=B^{\sigma+1}_{2,2}(\mathbb{R}^2;\mathbb{R})$. Equation \eqref{thmproof2} follows from the same identity and Proposition \ref{prop:Hspaces} \textit{(i)} together with the fact that $L^2([0,T];X)\subset L^t([0,T];X)$ if $1\leq t\leq 2$ by the Hölder inequality. By the previously stated results we easily get $v_i\in L^\rho([0,T]; B^{\sigma+2}_{2,2}(\mathbb{R}^2;\mathbb{R}))$. On the other hand $B^{\sigma+2}_{2,2}(\mathbb{R}^2;\mathbb{R})\subset B^{-M}_{\infty,\infty}(\mathbb{R}^2;\mathbb{R})$ for any $M>0$ \cite[Proposition 2.71]{BahouriCheminDanchin11}. This gives equation \eqref{thmproof3}. Similarly we get $\nabla \bm{v}\in  L^1([0,T];B^{\sigma+1}_{2,2}(\mathbb{R}^2;\mathbb{R}))$ and $B^{\sigma+1}_{2,2}(\mathbb{R}^2;\mathbb{R})\subset B^1_{2,\infty}(\mathbb{R}^2;\mathbb{R})$ (\cite[Proposition 2.71]{BahouriCheminDanchin11}. Furthermore, by  the Sobolev embeddings, we also get that $B^{\sigma+1}_{2,2}(\mathbb{R}^2;\mathbb{R})\subset L^\infty(\mathbb{R}^2;\mathbb{R})$.

Finally, since $A\neq 0$, we also have to show that
\[
\Vert (A\bm{f})(t)\Vert_{H^{\sigma+1}(\mathbb{R}^2;\mathbb{R}^3)}\leq \mathcal{A}(t)\Vert \bm{f}(t)\Vert_{H^{\sigma+1}(\mathbb{R}^2;\mathbb{R}^3)}.
\]
for some $\mathcal{A}\in L^1([0,T];\mathbb{R})$ (see Remark  3.17 in \cite{BahouriCheminDanchin11}). However since $\sigma+1>1$ we get
\[
\Vert (A\bm{f})(t)\Vert_{H^{\sigma+1}(\mathbb{R}^2;\mathbb{R}^3)}\lesssim \Vert A(t)\Vert_{H^{\sigma+1}(\mathbb{R}^2;\mathbb{R}^{3\times 3})}\Vert \bm{f}(t)\Vert_{H^{\sigma+1}(\mathbb{R}^2;\mathbb{R}^3)}.
\]
we see that $\mathcal{A}(t):=\Vert A(t)\Vert_{H^{\sigma+1}(\mathbb{R}^2;\mathbb{R}^{3\times 3})}$ is in $L^2([0,T];\mathbb{R})\subset L^1([0,T];\mathbb{R})$ using Proposition \ref{prop:Hspaces} \textit{(i)} because $A\in H^{\sigma+1}(\mathbb{R}^3;\mathbb{R}^{3\times 3})$. By Remark 3.20 in \cite{BahouriCheminDanchin11} we also need to approximate $A$ by a sequence of smooth functions satisfying the same inequality. But due to the way $\mathcal{A}(t)$ was chosen this can be done by just approximating $A$ by smooth functions converging in norm.

It remains to show that this solution indeed lies in $H^{\sigma+1}([0,T]\times\mathbb{R}^2;\mathbb{R}^3)$.
We consider
\begin{equation}\label{transporttimederive}
\partial_t\bm{f}=-(\bm{v}\cdot\nabla)\bm{f}-A\bm{f}+\bm{g}.
\end{equation}
The components of $-(\bm{v}\cdot\nabla)\bm{f}$ consist of sums of terms of the form $v_i\partial_{x_i}f_j$. Since $v_i\in H^{\sigma+2}([0,T]\times\mathbb{R}^2;\mathbb{R})$, $\partial_{x_i}f_j\in L^2([0,T],H^{\sigma}(\mathbb{R}^2;\mathbb{R}))$ we can apply Proposition \ref{prop:Hspaces} \textit{(v)}. This gives $v_i\partial_{x_i}f_j\in L^2([0,T],H^{\sigma}(\mathbb{R}^2;\mathbb{R}))$ and hence $-(\bm{v}\cdot\nabla)\bm{f}\in L^2([0,T],H^{\sigma}(\mathbb{R}^2;\mathbb{R}^3))$. Similarly the components of $-A\bm{f}$ consist of sums of terms of the form $a_{ij}f_i$, where $a_{ij}\in H^{\sigma+1}([0,T]\times\mathbb{R}^2;\mathbb{R})$ and $f_i\in L^2([0,T];H^{\sigma}(\mathbb{R}^2;\mathbb{R}))$. This allows us to again apply Proposition \ref{prop:Hspaces} \textit{(v)}, giving $a_{ij}f_i\in L^2([0,T];H^{\sigma}(\mathbb{R}^2;\mathbb{R}))$ and hence $-A\bm{f}\in L^2([0,T];H^{\sigma}(\mathbb{R}^2;\mathbb{R}^3))$.

Moreover, the components of $\bm{g}$ are functions in $H^{\sigma}([0,T]\times\mathbb{R}^2;\mathbb{R})$, which by Proposition \ref{prop:Hspaces} \textit{(i)} is a subspace of $L^2([0,T];H^{\sigma}(\mathbb{R}^2;\mathbb{R}))$, so $\bm{g}\in L^2([0,T];H^{\sigma}(\mathbb{R}^2;\mathbb{R}^3))$. Altogether this gives us $\partial_t\bm{f}\in L^2([0,T];H^{\sigma}(\mathbb{R}^2;\mathbb{R}^3))$. This in turn means that \\$\bm{f}\in L^2([0,T];H^{1+\sigma}(\mathbb{R}^2;\mathbb{R}^3))\cap H^1([0,T];H^{\sigma}(\mathbb{R}^2;\mathbb{R}^3))$. By Proposition \ref{prop:Hspaces} \textit{(iii)} this means $\bm{f}\in H^{\sigma}([0,T];H^1(\mathbb{R}^2;\mathbb{R}^3))$. 

With this information at hand we can redo the analysis of equation \eqref{transporttimederive}. We still have $v_i\in H^{\sigma+2}([0,T]\times\mathbb{R}^2;\mathbb{R})$ but now $\partial_{x_i}f_j\in H^{\sigma}([0,T];L^2(\mathbb{R}^2;\mathbb{R}))$. By applying Proposition \ref{prop:Hspaces} \textit{(iv)} we get $v_i\partial_{x_i}f_j\in H^{\sigma}([0,T];L^2(\mathbb{R}^2;\mathbb{R}))$ and hence $-(\bm{v}\cdot\nabla) \bm{f}\in H^{\sigma}([0,T];L^2(\mathbb{R}^2;\mathbb{R}^3))$. Similarly $a_{ij}\in H^{\sigma+1}([0,T]\times\mathbb{R}^2;\mathbb{R})$ is still true, but now $f_i\in H^{\sigma}([0,T]; H^1(\mathbb{R}^2;\mathbb{R}))$. Applying Proposition \ref{prop:Hspaces} \textit{(ii)} we get $f\in H^{\sigma}([0,T]\times\mathbb{R})$ and hence, by theorem 1.4.4.2 in \cite{Grisvard2011}, $a_{ij}f_j\in H^{\sigma}([0,T]\times\mathbb{R}^2;\mathbb{R})$. By proposition \ref{prop:Hspaces} \textit{(i)} this means $A\bm{f}\in H^{\sigma}([0,T]; L^2(\mathbb{R}^2;\mathbb{R}^3))$. 

Finally, since $\bm{g}\in H^{\sigma}([0,T]\times\mathbb{R}^2;\mathbb{R}^3)$ we get by Proposition \ref{prop:Hspaces} \textit{(i)} that \\$\bm{g}\in H^{\sigma}([0,T]; L^2(\mathbb{R}^2;\mathbb{R}^3))$. Altogether this means that $\partial_t\bm{f}\in H^{\sigma}([0,T]; L^2(\mathbb{R}^2;\mathbb{R}^3))$. Thus $\bm{f}\in H^{\sigma+1}([0,T];L^2(\mathbb{R}^2;\mathbb{R}^3))$. Recalling that we started out with \\$\bm{f}\in C([0,T];H^{\sigma+1}(\mathbb{R}^2;\mathbb{R}^3))\subset L^2([0,T];H^{\sigma+1}(\mathbb{R}^2;\mathbb{R}^3))$ we get through Proposition \ref{prop:Hspaces} \textit{(ii)} that $\bm{f}\in H^{\sigma+1}([0,T]\times\mathbb{R}^2;\mathbb{R}^3)$.
\end{proof}
\subsection{Finding the Solution}
We begin by solving the problem under the additional assumption that the inflow set $\partial\Omega_-$ is flat. This is because there is only a slight difference between a inflow set that is flat and one that is not. We clarify the difference at the relevant stage of the proof in Remark \ref{Rem:nonflatinflow}.

We seek a solution $\bm{f}\in H^{\sigma+1}(\Omega;\mathbb{R}^3)$ to $(\mathcal{TP})$ where $\bm{v}\in H^{\sigma+2}(\Omega;\mathbb{R}^3)$ and $\bm{f}_0\in H^{\sigma}(\partial\Omega_-;\mathbb{R}^3)$ are given.
Furthermore, by Lemma \ref{Lemma:IntegralCurves}, we also know that $\vert\bm{v}(\bm{x})\vert^2>d^2$, where $d$ is some positive constant. By the Sobolev embedding theorems we have that $\bm{v}\in C^1(\overline{\Omega};\mathbb{R}^3)$. This means that it is easily shown that 
\begin{equation}\label{eq:v1pos1}
\bm{v}(\bm{x})\cdot\bm{v}(\bm{y})>\frac{d^2}{2}
\end{equation}
if $\vert \bm{x}-\bm{y}\vert$ is sufficiently small. Similarly, since $\bm{v}(\bm{y})\cdot\bm{n}(\bm{y})<-c$ for any $\bm{y}\in \partial\Omega_-$ it is easily shown that 
\begin{equation}\label{eq:v1pos2}
\bm{v}(\bm{x})\cdot\bm{n}(\bm{y})<-c/2
\end{equation}
given that $\vert \bm{x}-\bm{y}\vert$ is sufficiently small. We denote by $\epsilon$ a number small enough for both \eqref{eq:v1pos1} and \eqref{eq:v1pos2} to be satisfied if $\vert \bm{x}-\bm{y}\vert<\epsilon$.

Recall that the integral curves of the vector field $\bm{v}$ are given by $t\mapsto \bm{\Phi}(t,\bm{x}_0)$, where $\bm{\Phi}$ is the flow defined by
\begin{align*}
\frac{d}{dt}\bm{\Phi}(t,\bm{x}_0)&=\bm{v}(\bm{\Phi}(t,\bm{x}_0)),\\
\bm{\Phi}(0,\bm{x}_0)&=\bm{x}_0.
\end{align*}
From Lemma \ref{Lemma:IntegralCurves} it follows that the integral curves cover all of $\Omega$. Moreover every integral curve intersects $\partial\Omega_-$ in one point. This means that the function
\[
\bm{\Psi}(t,\bm{x})=\bm{\Phi}(t,\bm{x})\vert_{\bm{x}\in\partial\Omega_-}
\]
maps $M=\cup_{\bm{x}\in \partial\Omega_-}[0,S(\bm{x})]\times \{\bm{x}\}$ onto $\overline{\Omega}$. For later we also note that since \\$\bm{v}\in H^{\sigma+2}(\Omega;\mathbb{R}^3)\hookrightarrow C^1(\Omega;\mathbb{R})$ we get that $\bm{\Psi}(\bm{x},t)\in C^1(M;\mathbb{R}^3)$. Furthermore $\bm{\Psi}$ and its derivatives are bounded in supremum norm by a constant depending on $\Omega$ and $\bm{v}$. It is also possible to show that the derivatives are bounded below. This means that $\bm{\Psi}^{-1}$ and its derivatives are also bounded in supremum norm.

Let $\bm{y}$ be a point on the inflow set. This gives us an integral curve, $\Sigma=\bm{\Psi}([0,S(\bm{y})]\times \{\bm{y}\})$. We can cover this streamline by open balls with radius $\epsilon/2$ centered at the points $\bm{y}_i=\Psi(s_i,\bm{y})$, $i=0,1,...,n$, starting with $\bm{y}_0=\bm{y}$, chosen in such a way that $\bm{x}_i\in B_{i-1}$, where $B_{i-1}$ denotes the ball centered at $\bm{x}_{i-1}$. We also assume that $\epsilon$ is small enough for $B_i\cap \Sigma$ to consist of only one connected component.

If we then choose a sufficiently small neighbourhood $B_{\bm{y}}\subset \overline{\partial\Omega}_-$ of $\bm{y}$ then the set $T=\bm{\Psi}(\cup_{\bm{x}\in B_{\bm{y}}} [0,S(\bm{x})]\times \{\bm{x}\})$ satisfies $\Sigma\subset T\subset \cup_{i=0}^n B_i$. Moreover, we also choose $B_{\bm{y}}$ to be small enough for $\overline{\Pi_i\cap T}\subset B_{i-1}$, where $\Pi_i$ denotes the plane containing $\bm{y}_i$ with normal in the direction of $\bm{v}(y_i)$. The plane $\Pi_i$ divides $B_i$ in the half-spheres $B_i^-$ and $B_i^+$, where $B_i^+$ is the half-sphere which $\bm{v}(y_i)$ points into. We also choose $B_{\bm{y}}$ small enough for $B_i\cap T$ to only consist of one connected component, which we can do since $B_i\cap \Sigma$ only consists of one connected component.

In each of the balls $B_i$ we can express the equation in $(\mathcal{TP})$ using an orthonormal basis with the first
basis vector in the direction of $\bm{v}(\bm{y}_i)$, or $-\bm{n}$ if $i=0$. Due to \eqref{eq:v1pos1} and \eqref{eq:v1pos2} this means that $v_1(\bm{x})$ is positive for $\bm{x}\in B_i$. Hence we can divide the equation by $v_1(\bm{x})$. Moreover $1/v_1(\bm{x})\in H^{\sigma+2}(\Omega;\mathbb{R})$. This follows from  Theorem 10 in \cite{BourdaudSickel2011} and the fact that $1/v_1$ can be expressed as the composition of a smooth function and a function in $H^{\sigma+2}(\Omega;\mathbb{R})$.

We define $\bm{v}'=\left(\frac{v_2}{v_1},\frac{v_3}{v_1}\right)$, $\nabla'=(\partial_{x_1},\partial_{x_2})$,
and $A=-\frac{1}{v_1}D\bm{v}$. Here $D\bm{v}$ denotes the Jacobian matrix of $\bm{v}$. From the properties of $\bm{v}$ and $1/v_1$ it follows that $\bm{v}'\in H^{\sigma+2}(\Omega;\mathbb{R}^2)$ and $A\in H^{\sigma+1}(\Omega;\mathbb{R}^{3\times 3})$. Dividing the equation in $(\mathcal{TP})$ by $v_1$ gives
\[
\partial_{x_1}\bm{f}+(\bm{v}'\cdot\nabla')\bm{f}+A\bm{f}=0.
\]
Hence starting with $B_0$ we want to solve
\begin{equation}\label{eq:transpB1}
\begin{aligned}
\partial_{x_1}\bm{f}+(\bm{v}'\cdot\nabla')\bm{f}+A\bm{f}&=0 &&\inn B_0\cap T,\\
\bm{f}&=\bm{f}_0 && \onn B_{\bm{y}}.
\end{aligned}
\end{equation}
If we put the origin at $\bm{y}$ then $B_{\bm{y}}$ is contained in the plane given by $x_1=0$ and $B_1\cap T$ is contained in the subset of $\mathbb{R}^3$ defined by $x_1\in[0,\epsilon/2]$. Then we can extend $\bm{v'}$, $A$ and $\bm{f}_0$ to $H^{\sigma+2}([0,\epsilon/2]\times\mathbb{R}^2;\mathbb{R}^3)$, $H^{\sigma+1}([0,\epsilon/2]\times\mathbb{R}^2;\mathbb{R}^{3\times3})$ and $H^{\sigma+1}(\mathbb{R}^2;\mathbb{R}^3)$ respectively. Applying Proposition \ref{prop:transport1} gives us a solution $\bm{f}\in H^{\sigma+1}(B_0\cap T;\mathbb{R}^3)$. We note that this solution is independent of the way $\bm{v}'$, $A$ and $\bm{f}_0$ are extended.
\begin{remark}\label{Rem:nonflatinflow} In the case when the inflow set is not flat we need to change variables in $B_0$ in such a way that $B_{\bm{y}}$ is mapped onto some subset of $\mathbb{R}^2$ to apply Proposition \ref{prop:transport1}. The boundary $\partial\Omega_-$ is sufficiently smooth for this to be done without changing the regularity of the functions involved. However, it does change the expressions slightly. Most importantly the function which we have to divide by, but by choosing $B_{\bm{y}}$ sufficiently small we can always make sure that this function remains bounded below by some constant.
\end{remark}

To proceed we assume we already have a solution $\tilde{\bm{f}}_{i-1}$ in $\cup_{k=0}^{i-1}B_{k}\cap T$ with $\tilde{\bm{f}}_{i-1}=\bm{f}_0$ on $B_{\bm{y}}$. Define $\eta_i$ and $\phi_i$ as smooth functions on $\cup_{k=0}^i B_k\cap T$ satisfying $\eta_i+\phi_i= 1$ on the whole domain, $\eta_i= 1$ on $\cup_{k=0}^{i-1}B_{k}\cap T\setminus B_i^+$ and $\phi_i=1$ on $B_i^+\cap T\setminus \cup_{k=0}^{i-1}B_{k}$. The sets $\cup_{k=0}^{i-1}B_{k}\cap T\setminus B_i^+$ and $B_i^+\cap T\setminus \cup_{k=0}^{i-1}B_{k}$ are always separated due to $\overline{\Pi_i\cap T}\subset B_{i-1}$ (See Figure \ref{Fig2}). This means we can always define the functions $\eta_i$ and $\phi_i$ this way. We consider the problem
\begin{figure}
\begin{center}
\includegraphics[scale=1.2]{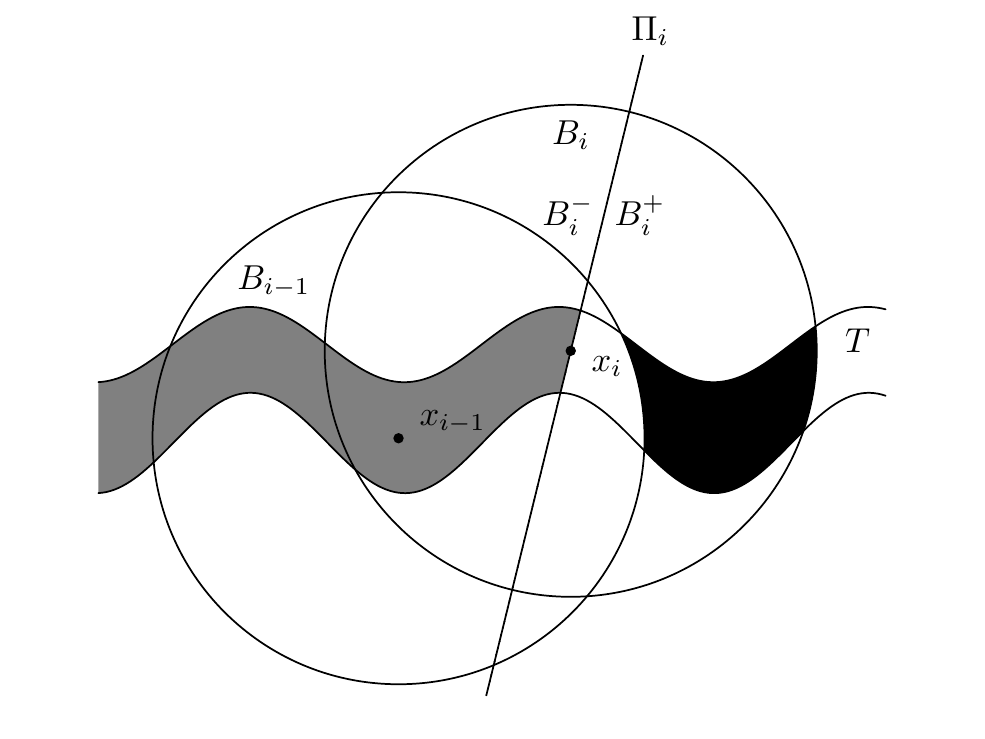}
\caption{A part of $T$ and the balls $B_{i-1}$ and $B_i$. The grey area is $\cup_{k=0}^{i-1}B_{k}\cap T\setminus B_i^+$ and the black area is $B_i^+\cap T\setminus \cup_{k=0}^{i-1}B_{k}$.}\label{Fig2}
\end{center}
\end{figure}
\begin{equation}\label{eq:transpBi}
\begin{aligned}
(\bm{v}\cdot\nabla)\bm{f}&=(\bm{f}\cdot\nabla)\bm{v}+\tilde{\bm{f}}_{i-1}(\bm{v}\cdot\nabla)\phi_i && \inn B_i^+\cap T,\\
\bm{f}&=0 &&\onn \Pi_i\cap T.
\end{aligned}
\end{equation}
The term $\tilde{\bm{f}}_{i-1}(\bm{v}\cdot\nabla)\phi_i$ is extended by $0$ in $B_i^+\cap T\setminus\cup_{k=0}^{i-1}B_{k}$. This can be done without problem since $\phi_i$ is constant there.
Expressed with respect to the coordinates related to $\Pi_i$ and extending the given functions \eqref{eq:transpBi} becomes
\begin{equation}\label{eq:transpBi2}
\begin{aligned}
\partial_{x_1}\bm{f}+(\bm{v}'\cdot\nabla')\bm{f}+A\bm{f}&=\frac{1}{v_1}\tilde{\bm{f}}_{i-1}(\bm{v}\cdot\nabla)\phi_i &&\text{for}\,0<x_1<\epsilon/2,\\
\bm{f}&=0 && \text{for}\, x_1=0.
\end{aligned}
\end{equation}
Since $\frac{1}{v_1}\tilde{\bm{f}}_{i-1}(\bm{v}\cdot\nabla)\phi_i\in H^{\sigma+1}([0,\epsilon/2]\times \mathbb{R}^2;\mathbb{R}^3)$ we get a solution $\bm{f}\in H^{\sigma+1}([0,\epsilon/2]\times \mathbb{R}^2;\mathbb{R}^3)$ by Proposition \ref{prop:transport1}. It is also clear that $\phi_i\tilde{\bm{f}}_{i-1}$ solves the problem \eqref{eq:transpBi} if restricted to $\cup_{k=0}^{i-1} B_k\cap B_i^+\cap T$. Uniqueness of the solutions to this problem is easily shown by considering the difference of two solutions. Hence $\bm{f}=\phi_i\tilde{\bm{f}}_{i-1}$ in $\cup_{k=0}^{i-1} B_k\cap B_i^+\cap T$. This means that if we extend $\eta_i\tilde{\bm{f}}_{i-1}$ and $\bm{f}$ by zero to $\cup_{k=0}^i B_k\cap T$ then $\tilde{\bm{f}}_i:=\eta_i\tilde{\bm{f}}_{i-1}+\bm{f}$ equals $\tilde{\bm{f}}_{i-1}$ if restricted to $\cup_{k=0}^{i-1} B_k\cap T$ and
\[
(\bm{v}\cdot\nabla)\tilde{\bm{f}}_i=(\tilde{\bm{f}}_i\cdot\nabla)\bm{v}
\]
in $\cup_{k=0}^i B_k\cap T$. So we have extended our solution to $\cup_{k=0}^i B_i\cap T$. Starting with the solution to \eqref{eq:transpB1} and repeating the argument a finite number of times gives a solution in all of $T$.

Now we can cover $\partial\Omega_-$ with a finite number, $m$, of open balls to get corresponding solutions in $T_j$, $j=1,...,m$. Then $\cup_{j=1}^m T_j$ covers $\Omega$. Moreover if $T_j\cap T_k\neq \emptyset$ we get that the solutions in $T_j$ and $T_k$ are equal since their difference is $0$ along any integral curve of $\bm{v}$. So if we define a partition of unity $\psi_j$, $j=1,...,m$ such that $\supp (\psi_j)\subset T_j$ and let $\bm{f}_j$ now denote the solution in $T_j$, then $\bm{f}=\sum_{j}^m\bm{f}_j\psi_j$ is a solution to $(\mathcal{TP})$.

\subsection{Estimates}
Below we prove the estimates \eqref{tpest1}, \eqref{tpest2} and \eqref{tpest3}.
\subsubsection{Higher Regularity Estimate}
For a solution of the transport equation given by \eqref{transport}
we have, from Theorem 3.14 in \cite{BahouriCheminDanchin11}, the estimate
\[
\begin{aligned}
\Vert \bm{f}\Vert_{L^\infty([0,T],H^{\sigma+1}(\mathbb{R}^2;\mathbb{R}^3))} &\lesssim \left(\Vert \bm{f}_0\Vert_{H^{\sigma+1}(\mathbb{R}^2;\mathbb{R}^3)}+\Vert \bm{g}\Vert_{L^1([0,T],H^{\sigma+1}(\mathbb{R}^2;\mathbb{R}^3))}\right)\\
&\qquad\times\exp{\left(C\int_0^T(\Vert \nabla \bm{v}(t)\Vert_{H^{\sigma}(\mathbb{R}^2;\mathbb{R}^3)}+\mathcal{A}(t))dt \right)}.
\end{aligned}
\]
This yields
\[
\Vert \bm{f}\Vert_{L^2([0,T],H^{\sigma+1}(\mathbb{R}^2;\mathbb{R}^3))}\lesssim\Vert \bm{f}_0\Vert_{H^{\sigma+1}(\mathbb{R}^2;\mathbb{R}^3)}+\Vert \bm{g}\Vert_{H^{\sigma+1}([0,T]\times\mathbb{R}^2;\mathbb{R}^3)}.
\]
Furthermore
\[
\Vert \bm{f}\Vert_{H^{\sigma}([0,T];H^1(\mathbb{R}^2))}\lesssim \Vert \bm{f} \Vert_{H^1([0,T];H^{\sigma}(\mathbb{R}^2))}+\Vert \bm{f}\Vert_{L^2([0,T];H^{\sigma+1}(\mathbb{R}^2))},
\]
where clearly
\[
\Vert \bm{f} \Vert_{H^1([0,T];H^{\sigma}(\mathbb{R}^2;\mathbb{R}^3))}\lesssim \Vert \bm{f} \Vert_{L^2([0,T];H^{\sigma}(\mathbb{R}^2;\mathbb{R}^3))}+\Vert \partial_t\bm{f} \Vert_{L^2([0,T];H^{\sigma}(\mathbb{R}^2;\mathbb{R}^3)).}
\]
If $\bm{f}$ satisfies the transport equation then
\[
\Vert \partial_t\bm{f} \Vert_{L^2([0,T];H^{\sigma}(\mathbb{R}^2))}\lesssim\Vert \bm{f} \Vert_{L^2([0,T];H^{\sigma+1}(\mathbb{R}^2))}+\Vert \bm{g} \Vert_{L^2([0,T];H^{\sigma}(\mathbb{R}^2))},
\]
which in combination with the estimates above yields
\[
\Vert \bm{f}\Vert_{H^{\sigma}([0,T];H^1(\mathbb{R}^2;\mathbb{R}^3))}\lesssim\Vert \bm{f}_0\Vert_{H^{\sigma+1}(\mathbb{R}^2;\mathbb{R}^3)}+\Vert \bm{g}\Vert_{H^{\sigma+1}([0,T]\times\mathbb{R}^2;\mathbb{R}^3)}.
\]
Similarly we have
\[
\Vert \bm{f}\Vert_{H^{\sigma+1}([0,T];L^2(\mathbb{R}^2;\mathbb{R}^3))}\lesssim \Vert \bm{f}\Vert_{L^2([0,T];L^2(\mathbb{R}^2;\mathbb{R}^3))}+\Vert \partial_t \bm{f}\Vert_{H^{\sigma}([0,T];L^2(\mathbb{R}^2;\mathbb{R}^3))},
\]
where
\begin{align*}
\Vert \partial_t \bm{f}\Vert_{H^{\sigma}([0,T];L^2(\mathbb{R}^2;\mathbb{R}^3))}\lesssim \Vert \bm{f}\Vert_{H^{\sigma}([0,T];H^1(\mathbb{R}^2;\mathbb{R}^3))}+\Vert \bm{g}\Vert_{H^{\sigma}([0,T];L^2(\mathbb{R}^2;\mathbb{R}^3))}.
\end{align*}
Combining this with the other estimates we get
\[
\Vert \bm{f}\Vert_{H^{\sigma+1}([0,T];L^2(\mathbb{R}^2;\mathbb{R}^3))}\lesssim \Vert \bm{f}_0\Vert_{H^{\sigma+1}(\mathbb{R}^2;\mathbb{R}^3)}+\Vert \bm{g}\Vert_{H^{\sigma+1}([0,T]\times\mathbb{R}^2;\mathbb{R}^3))}
\]
and thus
\[
\Vert \bm{f}\Vert_{H^{\sigma+1}([0,T]\times\mathbb{R}^2;\mathbb{R}^3)}\lesssim \Vert \bm{f}_0\Vert_{H^{\sigma+1}(\mathbb{R}^2;\mathbb{R}^3)}+\Vert \bm{g}\Vert_{H^{\sigma+1}([0,T]\times\mathbb{R}^2;\mathbb{R}^3))}.
\]

We wish to apply this to the equations \eqref{eq:transpB1} and \eqref{eq:transpBi2}. First we note that the operation of extending the given functions is continuous, which means that the norm of an extended function can be estimated by the norm of the restricted function. Moreover the $H^{\sigma+2}$-norm of $1/v_1$ can be estimated by some expression depending on the $H^{\sigma+2}$-norm of $\bm{v}$ and thus so can the $H^{\sigma+2}$-norm of $\bm{v}'$ and the $H^{\sigma+1}$-norm of $A$. Let $\bm{f}_1$ be the solution of \eqref{eq:transpB1}. The estimate above gives
\[
\Vert \bm{f}_1\Vert_{H^{\sigma+1}(B_1\cap T;\mathbb{R}^3)}\lesssim \Vert \bm{f}_0\Vert_{H^{\sigma+1}(\partial\Omega_-;\mathbb{R}^3)}.
\]
Let $\tilde{\bm{f}}_i$ be the solution of \eqref{eq:transpBi2} then
\[
\Vert \tilde{\bm{f}}_i\Vert_{H^{\sigma+1}(B_i^+\cap T;\mathbb{R}^3)}\lesssim\left\Vert \frac{1}{v_1}\bm{f}_{i-1}(\bm{v}\cdot\nabla)\phi_i\right\Vert_{H^{\sigma+1}(B_{i-1}\cap B^+_i\cap T;\mathbb{R}^3)}\lesssim \Vert \bm{f}_{i-1}\Vert_{H^{\sigma+1}(B_{i-1}\cap T;\mathbb{R}^3)}.
\]
Since $\bm{f}_i=\tilde{\bm{f}}_i+\eta_i\bm{f}_{i-1}$ this means
\[
\Vert \bm{f}_i\Vert_{H^{\sigma+1}(\cup_{k=1}^i B_k\cap T;\mathbb{R}^3)}\lesssim\Vert \bm{f}_{i-1}\Vert_{H^{\sigma+1}(\cup_{k=1}^{i-1} B_k\cap T;\mathbb{R}^3)}.
\]
With repeated use and combined with the estimate for $\bm{f}_1$ this yields
\[
\Vert \bm{f}_j\Vert_{H^{\sigma+1}(T_j;\mathbb{R}^3)}\lesssim \Vert \bm{f}_0\Vert_{H^{\sigma+1}(\partial\Omega_-;\mathbb{R}^3)},
\]
for the solution $\bm{f}_j$ in $T_j$.
Since the solution, $\bm{f}$, to $(\mathcal{TP})$ is a finite sum of functions like $\bm{f}_j$, we get
\[
\Vert \bm{f}\Vert_{H^{\sigma+1}(\Omega;\mathbb{R}^3)}\lesssim \Vert \bm{f}_0\Vert_{H^{\sigma+1}(\partial\Omega_-;\mathbb{R}^3)},
\]
which is \eqref{tpest1}.
\subsubsection{Difference of solutions}
Let $\bm{v}^{(1)}$ and $\bm{v}^{(2)}$ be two different velocity fields with corresponding solutions $\bm{f}^{(1)}$ and $\bm{f}^{(2)}$. Denote by $[\bm{v}]=\bm{v}^{(2)}-\bm{v}^{(1)}$ and similarly $[\bm{f}]=\bm{f}^{(2)}-\bm{f}^{(1)}$. Then
\[
(\bm{v}^{(1)}\cdot \nabla)[\bm{f}]=([\bm{f}]\cdot \nabla)\bm{v}^{(1)}+(\bm{f}^{(2)}\cdot\nabla)[\bm{v}]-([\bm{v}]\cdot \nabla)\bm{f}^{(2)}
\]
in $\Omega$ and $[\bm{f}]=0$ on $\partial\Omega_-$.
This means that $[\bm{h}](t,\bm{x}):=[\bm{f}](\bm{\Psi}(t,\bm{x}))$ satisfies
\begin{align*}
\partial_t[\bm{h}]&=A^{(1)}\circ\bm{\Psi}[\bm{h}]+\left((\bm{f}^{(2)}\cdot\nabla)[\bm{v}]-([\bm{v}]\cdot \nabla)\bm{f}^{(2)}\right)\circ \bm{\Psi} && \text{ for } 0<t<S(\bm{x}),\\
[\bm{h}]&=0 && \text{ for } t=0,
\end{align*}
where $\bm{\Psi}$ is the flow of $\bm{v}^{(1)}$. By applying Gröwall's inequality we get
\begin{align*}
\sup_{t}\vert [\bm{h}](t,\bm{x})\vert\lesssim \int_0^{S(x)} \vert \left((\bm{f}^{(2)}\cdot \nabla)  [\bm{v}]-([\bm{v}]\cdot \nabla)\bm{f}^{(2)}\right)\circ\bm{\Psi}(t,\bm{x})\vert dt
\end{align*}
and hence
\begin{align*}
\int_{\partial\Omega_-}\int_0^{S(x)} \vert [\bm{h}](t,\bm{x})\vert^2dtd\bm{x}\lesssim\int_{\partial\Omega_-}\int_0^{S(x)} \vert \left((\bm{f}^{(2)}\cdot \nabla) [\bm{v}]-([\bm{v}]\cdot \nabla)\bm{f}^{(2)}\right)\circ\bm{\Psi}\vert^2 dtd\bm{x}
\end{align*}
or equivalently
\[
\Vert [\bm{f}]\Vert_{L^2(\Omega;\mathbb{R}^3)} \lesssim\Vert(\bm{f}^{(2)}\cdot \nabla) [\bm{v}]-([\bm{v}]\cdot \nabla) \bm{f}^{(2)}\Vert_{L^2(\Omega;\mathbb{R}^3)},
\]
because $\bm{\Psi}$, $\bm{\Psi}^{-1}$ and their derivatives are bounded. This means that $\Vert\cdot\Vert_{L^2(\Omega;\mathbb{R}^3)}$ and $\Vert\cdot\Vert_{L^2(M;\mathbb{R}^3)}$ are equivalent if we change variables with $\bm{\Psi}$ and $\bm{\Psi}^{-1}$. Moreover
\[
\Vert(\bm{f}^{(2)}\cdot \nabla) [\bm{v}]-([\bm{v}]\cdot \nabla) \bm{f}^{(2)}\Vert_{L^2(\Omega;\mathbb{R}^3)}\lesssim\Vert \bm{f}_0\Vert_{H^{\sigma+1}(\partial\Omega_-;\mathbb{R}^3)}\Vert [\bm{v}]\Vert_{H^1(\Omega;\mathbb{R}^3)}.
\]
Combining these last two estimates gives
\[
\Vert [\bm{f}]\Vert_{L^2(\Omega;\mathbb{R}^3)} \lesssim\Vert \bm{f}_0\Vert_{H^{\sigma+1}(\partial\Omega_-;\mathbb{R}^3)}\Vert [\bm{v}]\Vert_{H^1(\Omega;\mathbb{R}^3)},
\]
which is \eqref{tpest3}.
\subsubsection{Lower Regularity Estimate}
Consider again a solution, $\bm{f}$, to $(\mathcal{TP})$. We set $\bm{h}(t,\bm{x})=\bm{f}(\bm{\Psi}(t,\bm{x}))$, which satisfies
\begin{align*}
& \partial_t \bm{h}=\left(A\bm{h}\right)\circ \bm{\Psi} && \text{ for }0<t<S(\bm{x}),\\
& \bm{h}=\bm{f}_0 && \text{ for } t=0.
\end{align*}
Grönwall's inequality gives us
\begin{equation}\label{eq:tpsolest}
\sup_t\vert \bm{h}(t,\bm{x})\vert\lesssim \vert \bm{f}_0(\bm{x})\vert
\end{equation}
This yields
\begin{align*}
\int_{\partial\Omega_-}\int_0^{S(x)} \vert \bm{h}(t,\bm{x})\vert^2dtd\bm{x} \lesssim \int_{\partial\Omega_-}\vert \bm{f}_0(\bm{x})\vert^2 d\bm{x}
\end{align*}
or equivalently
\[
\Vert \bm{f}\Vert_{L^2(\Omega)}\lesssim\Vert \bm{f}_0\Vert_{L^2(\partial\Omega_-)},
\]
which is \eqref{tpest2}. The inequality \eqref{eq:tpsolest} together with part $(iii)$ of Lemma \ref{Lemma:IntegralCurves} also gives us that the solution satisfies $\bm{f}\vert_{\partial\partial\Omega_+}=0$ if $\bm{f}_0\vert_{\partial\partial\Omega_-}=0$.
\begin{remark} It would be simpler if could use this change of variables to prove existence and the other estimates. However, this is not possible because then the solution in the original coordinates is $\bm{h}\circ \bm{\Psi^{-1}}$ and there currently exists no result which the author is aware of that in general gives us the desired regularity for this type of composition. Bourdaud and Sickel \cite{BourdaudSickel2011} gives a good overview of existing results.
\end{remark}

\section{Irrotational Solutions}\label{sec:Irrotational}
Since the main theorem relies on an irrotational solution $(\bm{v}_0,p_0)\in H^{\sigma+2}(\Omega;\mathbb{R}^3)\times H^{\sigma+2}(\Omega;\mathbb{R})$ to equations \eqref{eq:Euler1a}--\eqref{eq:Eulerbdry} we include a brief discussion of that problem here. Because this is not the main topic we restrict ourselves to a relatively easy special case; the case when $\Omega=U\times (0,L)$ for some open set $U\subset \mathbb{R}^2$, with $\partial\Omega_-=U\times \{0\}$ and $\partial\Omega_+=U\times \{L\}$. Since the solution is required to be irrotational it can be rewritten in terms of a scalar potential, $\psi$, for $\bm{v}_0$. It is well known that the equations \eqref{eq:Euler1a}--\eqref{eq:Eulerbdry} are satisfied by $\bm{v}_0=\nabla \psi$ and $p_0=\frac{1}{2}\vert \nabla \psi \vert^2$ if $\psi$ satisfies
\begin{equation}\label{eq:potential}
\begin{aligned}
\Delta\psi&=0 &&\inn \Omega,\\
\bm{n}\cdot \nabla \psi&=\phi &&\onn \partial\Omega.
\end{aligned}
\end{equation}
That we can find a potential with the required properties given by the following proposition.
\begin{proposition} If $\phi$ is a given function satisfying, $\phi\vert_{\partial\Omega_\pm}\in H^{\sigma+3/2}(\partial\Omega_\pm,\mathbb{R})$, $\phi\vert_{\partial\Omega_0}\equiv 0$ and $n_\pm\cdot\nabla\phi=0$, where $n_\pm$ denotes the normal of $\partial\partial\Omega_\pm$, then there exists a solution $\psi\in H^{\sigma+3}(\Omega;\mathbb{R})$ to $\eqref{eq:potential}$. Moreover, if $\phi\vert_{\partial\Omega_+}>c$ and $\phi \vert_{\partial\Omega_-}<-c$ for some constant $c>0$, then the integral curves of $\nabla \psi$ have no stagnation points and $\nabla \psi$ has no closed integral curves.
\end{proposition}
\begin{proof}
We can find the solution using the same method as in section \ref{Sec:DCP} given that we have imposed the compatibility conditions $n_\pm\cdot\nabla\phi=0$. To show that the solution has no stagnation points and no closed integral curves we use the fact that the third component of $\bm{v}_0$  is given by $\partial_{x_3}\psi$ which solves
\[
\begin{aligned}
\Delta\partial_{x_3}\psi&=0 && \inn \Omega,\\
\partial_{x_3}\psi&=-\phi && \onn \partial\Omega_-,\\
\partial_{x_3}\psi&=\phi && \onn \partial\Omega_+,\\
n\cdot\nabla\partial_{x_3}\psi&=0 && \onn \partial\Omega_0.
\end{aligned}
\]
By the maximum principle \cite[Section $6.4.2$]{Evans2010} $\partial_{x_3}\psi$ will not attain its minimum at an interior point. Furthermore by Hopf's lemma \cite[Section $6.4.2$]{Evans2010} if this minimum is attained at a boundary point, except the points at the edges $\partial\partial\Omega_+$ and $\partial\partial\Omega_-$, the normal derivative must be negative. Hence this minimum is not attained on $\partial\Omega$ and thus must be attained at either $\overline{\partial\Omega}_-$ or $\overline{\partial\Omega}_+$. Since $\phi$ is bounded below by $c>0$ on $\partial\Omega_+$ and $-\phi$ is bounded below on $\partial\Omega_-$ by the same constant $c$ it follows that $\partial_{x_3}\psi$ is bounded below in all of $\Omega$ by $c$. Thus $\bm{v}_0\neq 0$ since the third component is always positive. Moreover, this also implies that there are no closed integral curves. Indeed, if 
\[
\partial_t\bm{\Phi}(\bm{x}_0,t)=\bm{v_0}(\bm{\Phi}(t,\bm{x}_0))
\]
then $\Phi_3$ is monotonically increasing in $t$ and we can never have $\bm{\Phi}(t_1,\bm{x}_0)=\bm{\Phi}(t_2,\bm{x}_0)$ if $t_1\neq t_2$.
\end{proof}
\section*{Acknowledgment}\noindent{This project has received funding from the European Research Council (ERC) under the European Union's Horizon 2020 research and innovation programme (grant agreement no 678698). The author would like to thank Erik Wahlén for valuable discussions and proofreading.}
\newpage
\appendix
\section{Appendix}\label{appendix}
\begin{proposition}\label{prop:Hspaces}
Let $U$ be a subset of $\mathbb{R}^d$. Then we have the following:

(i) If $s_1, s_2\geq 0$ and $s=s_1+s_2$, then
\[
H^s(U\times\mathbb{R}^n;\mathbb{R})\hookrightarrow H^{s_1}(U;H^{s_2}(\mathbb{R}^n;\mathbb{R})).
\]

(ii) If $s\geq 0$, then
\[
H^s(U;L^2(\mathbb{R}^n;\mathbb{R}))\cap L^2(U;H^s(\mathbb{R}^n;\mathbb{R}))= H^s(U\times \mathbb{R}^n;\mathbb{R}).
\]

(iii) If $s_1'\leq s_1\leq s_1''$ and $s_2''\leq s_2\leq s_2'$ and $s_1+s_2=s_1'+s_2'=s_1''+s_2''$, then
\[
H^{s_1'}(U;H^{s_2'}(\mathbb{R}^n;\mathbb{R}))\cap H^{s_1''}(U;H^{s_2''}(\mathbb{R}^n;\mathbb{R}))\hookrightarrow H^{s_1}(U;H^{s_2}(\mathbb{R}^n;\mathbb{R})).
\]

(iv) If $r\geq 0$, $s>r+n/2$ and $s>n/2+d/2$, then for $u\in H^s(U\times \mathbb{R}^n;\mathbb{R})$ and $v\in H^r(U; L^2(\mathbb{R}^n;\mathbb{R}))$ we have $uv\in H^r(U; L^2(\mathbb{R}^n;\mathbb{R}))$ with
\[
\Vert uv\Vert_{H^r(U; L^2(\mathbb{R}^n;\mathbb{R}))}\lesssim \Vert u\Vert_{H^s(U\times \mathbb{R}^n;\mathbb{R})}\Vert v\Vert_{H^r(U; L^2(\mathbb{R}^n;\mathbb{R}))}
\]

(v) If $r\geq 0$, $s>r+d/2$ and $s>n/2+d/2$, then for $u\in H^s(U\times \mathbb{R}^n;\mathbb{R})$ and $v\in L^2(U;H^r(\mathbb{R}^n;\mathbb{R}))$ we have $uv\in L^2(U;H^r(\mathbb{R}^n;\mathbb{R}))$, with
\[
\Vert uv\Vert_{L^2(U;H^r(\mathbb{R}^n;\mathbb{R}))}\lesssim \Vert u\Vert_{H^s(U\times \mathbb{R}^n;\mathbb{R})}\Vert v\Vert_{L^2(U;H^r(\mathbb{R}^n;\mathbb{R}))}
\]
\end{proposition}
\begin{remark} If $X$ and $Y$ are Banach spaces we use the notation $X\hookrightarrow Y$ to denote that $X$ is continuously embedded in $Y$. I.e. if $x\in X$ then $x\in Y$ and there exists a constant $C$ such that $\Vert x\Vert_Y\leq C\Vert x\Vert_X$. By $X=Y$ we mean $X\hookrightarrow Y$ and $Y\hookrightarrow X$.
\end{remark}
\begin{proof} The case when $U=\mathbb{R}^d$ follows from the definition and some estimates we show below for parts \textit{(i)}--\textit{(iv)}. The esitmates for the last part are omitted because of the similarity to the estimates of part \textit{(iv)}. Theorem 4.1 in \cite{Amann2000} gives us the existence of an extension operator in $\mathcal{L}(H^{s_1}(U;H^{s_2}(\mathbb{R}^n;\mathbb{R}));H^{s_1}(\mathbb{R}^d;H^{s_2}(\mathbb{R}^n;\mathbb{R})))$. Proper application of this operator will give us the desired result.

To show the estimates in the case when $U=\mathbb{R}^d$ let $\bm{x}\in\mathbb{R}^d$ and $\bm{y}\in\mathbb{R}^n$, $\mathcal{F}_x$ be the Fourier transform $\bm{x}\to\bm{\xi}$, and $\mathcal{F}_y$ the Fourier transform $\bm{y}\to\bm{\eta}$.

\textit{(i)} If $u(\bm{x},\bm{y})\in H^s(\mathbb{R}^{d+n})$, then 
\begin{align*}
\Vert u\Vert_{H^{s_1}(\mathbb{R}^d,H^{s_2}(\mathbb{R}^n)}^2&=\int_{\mathbb{R}^d}\int_{\mathbb{R}^n} (1+\vert\bm{\xi}\vert^2)^{s_1}(1+\vert\bm{\eta}\vert^2)^{s_2}\vert\hat{u}(\bm{\xi},\bm{\eta})\vert^2d\bm{\eta}d\bm{\xi}\\
&\leq \int_{\mathbb{R}^d}\int_{\mathbb{R}^n} (1+\vert\bm{\xi}\vert^2+\vert\bm{\eta}\vert^2)^{s_1}(1+\vert\bm{\xi}\vert^2+\vert\bm{\eta}\vert^2)^{s_2}\vert\mathcal{F}_x\mathcal{F}_y u(\bm{\xi},\bm{\eta})\vert^2d\bm{\eta}d\bm{\xi}\\
&=\Vert u\Vert_H^s(\mathbb{R}^{d+n}).
\end{align*}

\textit{(ii)} We only need to show the embedding $H^s(\mathbb{R}^d;L^2(\mathbb{R}^n))\cap L^2(\mathbb{R}^d; H^s(\mathbb{R}^n))\hookrightarrow H^s(\mathbb{R}^{d+n})$ since the other follows immediately from part \textit{(i)}. \\If $u\in H^s(U;L^2(\mathbb{R}^n;\mathbb{R}))\cap L^2(U;H^s(\mathbb{R}^n;\mathbb{R}))$, then
\begin{align*}
\Vert u\Vert_{H^s(\mathbb{R}^{d+n})}&=\int_{\mathbb{R}^{d+n}}(1+\vert\bm{\xi}\vert^2+\vert\bm{\eta}\vert^2)^s\vert \mathcal{F}_x\mathcal{F}_y u(\bm{\xi},\bm{\eta})\vert^2d\bm{\xi}d\bm{\eta}\\
&\leq \int_{\mathbb{R}^{d+n}}2^s[(1+\vert\bm{\xi}\vert^2)^s+(1+\vert\bm{\eta}\vert^2)^s]\vert \mathcal{F}_x\mathcal{F}_y u(\bm{\xi},\bm{\eta})\vert^2d\bm{\xi}d\bm{\eta}\\
&\lesssim \Vert u\Vert_{H^s(\mathbb{R}^d;L^2(\mathbb{R}^n))}+\Vert u\Vert_{L^2(\mathbb{R}^d;H^s(\mathbb{R}^n))}.
\end{align*}

\textit{(iii)} If $u\in H^{s_1'}(U;H^{s_2'}(\mathbb{R}^n;\mathbb{R}))\cap H^{s_1''}(U;H^{s_2''}(\mathbb{R}^n;\mathbb{R}))$, then 
\begin{align*}
\Vert u\Vert_{H^{s_1}(\mathbb{R}^d;H^{s_2}(\mathbb{R}^n))}&=\int_{\vert\bm{\xi}\vert\leq \vert\bm{\eta}\vert}(1+\vert\bm{\xi}\vert^2)^{s_1}(1+\vert \bm{\eta}\vert^2)^{s_2}\vert \mathcal{F}_x\mathcal{F}_y u(\bm{\xi},\bm{\eta})\vert^2d\bm{\eta}d\bm{\xi}\\
&\qquad+\int_{\vert\bm{\xi}\vert>\vert\bm{\eta}\vert}(1+\vert\bm{\xi}\vert^2)^{s_1}(1+\vert \bm{\eta}\vert^2)^{s_2}\vert \mathcal{F}_x\mathcal{F}_y u(\bm{\xi},\bm{\eta})\vert^2d\bm{\eta}d\bm{\xi}\\
&\leq \int_{\vert\bm{\xi}\vert\leq \vert\bm{\eta}\vert}(1+\vert\bm{\xi}\vert^2)^{s_1'}(1+\vert \bm{\eta}\vert^2)^{s_2+s_1-s_1'}\vert \mathcal{F}_x\mathcal{F}_y u(\bm{\xi},\bm{\eta})\vert^2d\bm{\eta}d\bm{\xi}\\
&\qquad+\int_{\vert\bm{\xi}\vert>\vert\bm{\eta}\vert}(1+\vert\bm{\xi}\vert^2)^{s_1+s_2-s_2''}(1+\vert \bm{\eta}\vert^2)^{s_2''}\vert \mathcal{F}_x\mathcal{F}_y u(\bm{\xi},\bm{\eta})\vert^2d\bm{\eta}d\bm{\xi}\\
&\leq\Vert u\Vert_{H^{s_1'}(\mathbb{R}^d;H^{s_2'}(\mathbb{R}^n))}+\Vert u\Vert_{H^{s_1''}(\mathbb{R}^d;H^{s_2''}(\mathbb{R}^n))}
\end{align*}

\textit{(iv)} If $u\in H^s(U\times \mathbb{R}^n;\mathbb{R})$ and $v\in H^r(U; L^2(\mathbb{R}^n;\mathbb{R}))$, then
\begin{align*}
\Vert uv\Vert_{H^{r}(\mathbb{R}^d;L^2(\mathbb{R}^n))}^2&=\int_{\mathbb{R}^n} \Vert uv(\cdot,\bm{y})\Vert_{H^r(\mathbb{R}^d)}^2d\bm{y}\\
&\lesssim \int_{\mathbb{R}^n} \Vert u(\cdot,\bm{y})\Vert_{H^t(\mathbb{R}^d)}^2\Vert v(\cdot,\bm{y})\Vert_{H^r(\mathbb{R}^d)}^2d\bm{y}\\
&\lesssim \sup_{\bm{y}\in \mathbb{R}^n} \Vert u(\cdot,\bm{y})\Vert_{H^t(\mathbb{R}^d)}^2 \Vert u\Vert_{H^r(\mathbb{R}^d;L^2(\mathbb{R}^n))}^2
\end{align*}
for any $t\geq r$ and $t>d/2$. Set $\sigma=s-t$ and note that
\begin{align*}
\vert \mathcal{F}_x u(\bm{\xi},\bm{y})\vert&=\vert \mathcal{F}_y^{-1}\mathcal{F}_x\mathcal{F}_y u (\bm{\xi},\bm{y})\vert\\
&\lesssim \int_{\mathbb{R}^n}(1+\vert \bm{\eta}\vert^2)^{-\sigma /2}(1+\vert \bm{\eta}\vert^2)^{\sigma /2}\vert \mathcal{F}_x\mathcal{F}_y u(\bm{\xi},\bm{\eta})\vert d\bm{\eta}\\
&\lesssim \left(\int_{\mathbb{R}^n} (1+\vert \bm{\eta}\vert^2)^{-\sigma}d\bm{\eta}\right)^{1/2}\left(\int_{\mathbb{R}^n}(1+\vert \bm{\eta}\vert^2)^{\sigma}\vert \mathcal{F}_x\mathcal{F}_y u(\bm{\xi},\bm{\eta})\vert^2 d\bm{\eta}\right)^{1/2}\\
&\lesssim \left(\int_{\mathbb{R}^n}(1+\vert \bm{\eta}\vert^2)^{\sigma}\vert \mathcal{F}_x\mathcal{F}_y u(\bm{\xi},\bm{\eta})\vert^2 d\bm{\eta}\right)^{1/2}
\end{align*}
if $s-t=\sigma>n/2$. Under the assumptions of the proposition we can choose a $t$ such that all the inequalities involving $t$ are satisfied. It follows that
\begin{align*}
\Vert u(\cdot,\bm{y})\Vert_{H^t(\mathbb{R}^d)}&=\int_{\mathbb{R}^d}(1+\vert\bm{\xi}\vert^2)^t\vert \mathcal{F}_x u(\bm{\xi},\bm{y})\vert^2d\bm{\xi}\\
&\lesssim \int_{\mathbb{R}^d}(1+\vert\bm{\xi}\vert^2)^t \int_{\mathbb{R}^n}(1+\vert \bm{\eta}\vert^2)^{\sigma}\vert \mathcal{F}_x\mathcal{F}_y u(\bm{\xi},\bm{\eta})\vert^2 d\bm{\xi} d\bm{\eta}\\
&\leq \int_{\mathbb{R}^{d+n}}(1+\vert\bm{\xi}\vert^2+\vert\bm{\eta}\vert^2)^{t+\sigma}\vert \mathcal{F}_x\mathcal{F}_y u(\bm{\xi},\bm{\eta})\vert^2 d\bm{\xi} d\bm{\eta}\\
&=\Vert u\Vert_{H^s(\mathbb{R}^{n+d})},
\end{align*}
which gives the desired estimate.
\end{proof}
\newpage
\bibliographystyle{siam}
\bibliography{References}
\end{document}